\documentclass[12pt]{article}
\pdfoutput=1

\usepackage{amsmath}
\usepackage[utf8]{inputenc}
\usepackage[UKenglish]{babel}
\usepackage [autostyle]{csquotes}
\MakeOuterQuote{"}
\usepackage[normalem]{ulem} 
\usepackage{aliascnt}
\usepackage{amssymb}
\usepackage{amsmath}
\usepackage{amsthm}
\usepackage{faktor}
\usepackage{mathtools}
\usepackage{amsfonts}
\usepackage{enumerate}
\usepackage{enumitem}
\usepackage{multicol}
\usepackage{float}
\usepackage{tikz}
\usepackage[all]{xy}
\usepackage{tikz-cd} 
\usepackage{comment}
\usepackage[round,comma,sort
]{natbib}
\usepackage{hyperref}
\hypersetup{
    colorlinks=true,       
    linkcolor=blue,          
    citecolor=olive,       
    filecolor=magenta,      
    urlcolor=blue          
}
\usepackage[capitalise, nameinlink]{cleveref}
\crefname{subsection}{subsection}{subsections}
\newcommand{\Z}{\mathbb{Z}}

\newcommand{\gp}[2]{\langle #1 \, | \, #2 \rangle}
\newcommand{\sgp}[1]{\langle #1 \rangle}

\DeclareMathOperator{\Ker}{Ker}

\def\coloneqq{\mathrel{\mathop\mathchar"303A}\mkern-1.2mu=}

\begin{document}

\title{Droms Theorems for twisted right-angled Artin groups}
\author{Simone Blumer, Islam Foniqi, Claudio Quadrelli}

%\date{\today}
\date{}
\maketitle

% ------- Theorem styles -------
\theoremstyle{plain}
\newtheorem{theorem}{Theorem}[section]
\newtheorem*{theorem*}{Theorem}
\newtheorem{conj}[theorem]{Conjecture}
\newtheorem{notation}[theorem]{Notation}
\newtheorem*{note}{Note}
% definition, example
\theoremstyle{definition}

\setlength{\parindent}{0em}
\setlength{\parskip}{0.5em} 
\author{}

\newaliascnt{conjecture}{theorem}
\newtheorem{conjecture}[conjecture]{Conjecture}
\aliascntresetthe{conjecture}
\providecommand*{\conjectureautorefname}{Conjecture}

\newaliascnt{lemma}{theorem}
\newtheorem{lemma}[lemma]{Lemma}
\aliascntresetthe{lemma}
\providecommand*{\lemmaautorefname}{Lemma}

\newaliascnt{proposition}{theorem}
\newtheorem{proposition}[proposition]{Proposition}
\aliascntresetthe{proposition}
\providecommand*{\propositionautorefname}{Proposition}

\newaliascnt{cor}{theorem}
\newtheorem{cor}[cor]{Corollary}
\aliascntresetthe{cor}
\providecommand*{\corautorefname}{Corollary}

\newaliascnt{claim}{theorem}
\newtheorem{claim}[claim]{Claim}

\newaliascnt{notation}{theorem}
\aliascntresetthe{notation}
\providecommand*{\notationautorefname}{Notation}

\aliascntresetthe{claim}
\providecommand*{\claimautorefname}{Claim}

\newaliascnt{remark}{theorem}
\newtheorem{remark}[remark]{Remark}
\aliascntresetthe{remark}
\providecommand*{\remarkautorefname}{Remark}

\newtheorem*{claim*}{Claim}
\theoremstyle{definition}

\newaliascnt{definition}{theorem}
\newtheorem{definition}[definition]{Definition}
\aliascntresetthe{definition}
\providecommand*{\definitionautorefname}{Definition}

\newaliascnt{example}{theorem}
\newtheorem{example}[example]{Example}
\aliascntresetthe{example}
\providecommand*{\exampleautorefname}{Example}

\newaliascnt{question}{theorem}
\newtheorem{question}[question]{Question}
\aliascntresetthe{question}
\providecommand*{\questionautorefname}{Question}

\begin{abstract}
We characterize twisted right-angled Artin groups whose finitely generated subgroups are also twisted right-angled Artin groups. Additionally, we give a classification of coherence within this class of groups in terms of the defining graph. Furthermore, we provide a solution to the isomorphism problem for a notable subclass of these groups.
\end{abstract}

\renewcommand{\thefootnote}{\fnsymbol{footnote}} 

\thefootnote{
}

\footnotetext{The second-named author acknowledges support from the EPSRC Fellowship grant EP/V032003/1 ‘Algorithmic, topological and geometric aspects of infinite groups, monoids and inverse semigroups’.
The third-named author acknowledges his membership to the national group GNSAGA of the Italian National Institute of Advanced Mathematics "F. Severi".
}

%%%%%%%%%%%%%%%%%%%%%%%%%%%%%%%%%%%%%%%%%%%%%%%%%%%%%%%%%%%%%%%%%%%%%%%%%%%%%%%%%%%%%%%%%%%%%%%%%%%%%%%%

\section{Introduction}\label{Introduction}

In geometric group theory, the class of \emph{right-angled Artin groups} (abbreviated as \emph{RAAGs}) plays a significant role; see \citep{charney2007introduction} for a comprehensive survey. RAAGs, also known as \emph{graph groups}, can be represented by simplicial graphs, where each RAAG is associated with a specific simplicial graph that encodes its defining relations.
A generalisation of these, known as \emph{twisted right-angled Artin groups} (abbreviated as \emph{T-RAAGs}), first appeared in \citep{pride1986tits}, and later in \citep{clancy2010homology} -  where the concept of a \emph{twisted Artin group} was introduced. We also refer to T-RAAGs as \emph{twisted graph groups}.
In spite of the rather elementary presentation (see below), T-RAAGs yield surprising variety and flexibility: they span from free groups to free abelian and meta-abelian groups and, unlike RAAGs, they may admit non-trivial torsion.
They have been the subject of several recent works, see, e.g., \citep*{foniqi2022,foniqi2024twisted, foniqi2024subgroup}, \citep*{antolin2025traag}, \citep*{himeno2024twisted}.

T-RAAGs are defined by presentations consisting of generators and relations: there are finitely many generators, and for each pair of distinct generators~$a, b$, there is at most one relation between them. This relation is either a commutation~$ab = ba$ or a so-called \emph{Klein relation}~$aba = b$. Such presentations can be conveniently encoded using \emph{mixed graphs}, where vertices correspond to the generators, and edges represent the relations. An undirected edge between~$a$ and~$b$ indicates the commutation relation~$ab = ba$, while a directed edge from~$a$ to~$b$ specifies the Klein relation~$aba = b$; if two generators~$a$ and~$b$ are not connected by an edge, no relation is imposed between them (for the formal definition of a mixed graph see~\Cref{ssec:mixed graphs}).
If~$\Gamma$ is a mixed graph, the associated T-RAAG is denoted by~$T(\Gamma)$.

Two basic and significant examples of T-RAAGs are the fundamental groups of the torus and the Klein bottle: the torus group is~$\Z^2 = \gp{a, b}{ab = ba}$, while the Klein bottle group is~$K = \gp{a, b}{aba = b}$.
The associated mixed graphs are, respectively,
\begin{figure}[H]
\centering
\begin{tikzpicture}[>={Straight Barb[length=7pt,width=6pt]},thick]
\draw[] (-0.5, 0) node[left] {$\mathrm{P}_2 =~$};
\draw[fill=black] (0,0) circle (1.5pt) node[left] {$a$};
\draw[fill=black] (2,0) circle (1.5pt) node[right] {$b\,\,,$};

\draw[] (5.5, 0) node[left] {$\Xi_2 =~$};
\draw[fill=black] (6,0) circle (1.5pt) node[left] {$a$};
\draw[fill=black] (8,0) circle (1.5pt) node[right] {$b\,\,,$};

\draw[thick] (0.1,0) -- (1.9,0);
\draw[thick, ->] (6.1,0) -- (7.9,0);
\end{tikzpicture}
\end{figure}
so that~$\Z^2=T(\mathrm{P}_2)$ and~$K=T(\Xi_2)$.
It is important to note that~$K$ is not a RAAG, as RAAGs are bi-orderable~\citep*{duchamp1992lower}, whereas~$K$ is not.

In the late '80s, Carl Droms obtained several remarkable results on the structure of RAAGs.
The aim of this work is to prove the "twisted analogue" of some of Droms' results.

One of the most renowned of such results is the theorem which characterises those RAAGs whose finitely generated subgroups are again RAAGs.
These RAAGs are precisely those whose associated simplicial graph does not contain as induced subgraph any of the following:
\begin{figure}[H]
\centering
\begin{tikzpicture}[>={Straight Barb[length=7pt,width=6pt]},thick, scale =0.75]
\draw[] (-6.5, 0) node[left] {$\mathrm{P}_4 =~$};
\draw[fill=black] (-6,0) circle (1pt) node[above] {$a$};
\draw[fill=black] (-5,0) circle (1pt) node[above] {$b$};
\draw[fill=black] (-4,0) circle (1pt) node[above] {$c$};
\draw[fill=black] (-3,0) circle (1pt) node[above] {$d$};
\draw[] (-3, 0) node[right] {$\,\,\, ,$};

\draw[thick] (-5.9,0) -- (-5.1,0);
\draw[thick] (-4.9,0) -- (-4.1,0);
\draw[thick] (-3.9,0) -- (-3.1,0);

\draw[] (1, 0) node[left] {$\mathrm{C}_4 =~$};
\draw[fill=black] (2,1) circle (1pt) node[left] {$a$};
\draw[fill=black] (4,1) circle (1pt) node[right] {$b$};
\draw[fill=black] (4,-1) circle (1pt) node[right] {$c$};
\draw[fill=black] (2,-1) circle (1pt) node[left] {$d$};
\draw[] (4.5, 0) node[right] {$,$};

\draw[thick] (2.1, 1) -- (3.9, 1);
\draw[thick] (4,0.9) -- (4,-0.9);
\draw[thick] (2.1,-1) -- (3.9,-1);
\draw[thick] (2,-0.9) -- (2,0.9);
\end{tikzpicture}\label{fig:P4C4}
\end{figure}
see \citep{droms1987subgroups}.
Simplicial graphs satisfying this property are called \emph{Droms (simplicial) graphs}. It is worth remarking that these graphs, and their associated groups~$L = A(P_4)$ and~$A(C_4) \cong F_2 \times F_2$, are sometimes referred to as \emph{poisonous}, as they frequently serve as obstructions to certain group-theoretic properties.

In order to state the "twisted analogue" of the aforementioned result, we need the notion of \emph{special mixed graph}.
Roughly speaking, a mixed graph~$\Gamma$ is \emph{special} if the following holds: if a vertex~$b$ is the tip of a directed edge, then every edge joining the vertex~$b$ to another vertex is a directed edge "pointing" at~$b$.
In special mixed graphs, we call terminal vertices of directed edges \emph{sinkholes}.
For example, consider the mixed graphs
\begin{figure}[H]
\centering
\begin{tikzpicture}[>={Straight Barb[length=7pt,width=6pt]},thick, scale =0.75]

\draw[] (-5.8, 0) node[left] {$\Gamma_1 =~$};
\draw[fill=black] (-5,1) circle (1pt) node[left] {$a$};
\draw[fill=black] (-3,1) circle (1pt) node[right] {$b$};
\draw[fill=black] (-3,-1) circle (1pt) node[right] {$c$};
\draw[fill=black] (-5,-1) circle (1pt) node[left] {$d$};
\draw[] (-2.5, 0) node[right] {$,$};

\draw[thick, ->] (-4.9, 1) -- (-3.1, 1);
\draw[thick, ->] (-3,-0.9) -- (-3,0.9);
\draw[thick] (-4.9,-0.9) -- (-3.1,0.9);
\draw[thick, ->] (-5,-0.9) -- (-5,0.9);

\draw[] (1.3, 0) node[left] {$\Gamma_2 =~$};
\draw[fill=black] (2,1) circle (1pt) node[left] {$a'$};
\draw[fill=black] (4,1) circle (1pt) node[right] {$b'$};
\draw[fill=black] (4,-1) circle (1pt) node[right] {$c'$};
\draw[fill=black] (2,-1) circle (1pt) node[left] {$d'$};
\draw[] (4.5, 0) node[right] {$.$};

\draw[thick, ->] (2.1, 1) -- (3.9, 1);
\draw[thick, ->] (4,-0.9) -- (4,0.9);
\draw[thick] (2,-0.9) -- (2,0.9);
\draw[thick, ->] (2.1,-0.9) -- (3.9,0.9);
\end{tikzpicture}
\end{figure}
Then~$\Gamma_1$ is not special; while~$\Gamma_2$ is special, and the vertex~$b$ is its only sinkhole.
We define the "mixed analogue" of Droms simplicial graphs as follows.

\begin{definition}\label{def:droms mixed graphs}
A mixed graph~$\Gamma$ is said to be a \emph{Droms mixed graph} if the following hold:
 \begin{itemize}
  \item[{\rm (i)}] the underlying simplicial graph~$\overline{\Gamma}$ -- i.e., the simplicial graph obtained "forgetting" the direction of directed edges of~$\Gamma$ -- is a Droms simplicial graph;
  \item[{\rm (ii)}]~$\Gamma$ is special;
  \item[{\rm (iii)}]~$\Gamma$ does not contain the special mixed graph
  \begin{figure}[H]
\centering
\begin{tikzpicture}[>={Straight Barb[length=7pt,width=6pt]},thick]
\draw[] (-0.5, 0) node[left] {$\Lambda_{\mathrm s} =~$};
\draw[fill=black] (0,0) circle (1.5pt) node[left] {$a_1$};
\draw[fill=black] (2,0) circle (1.5pt) node[above] {$b$};
\draw[fill=black] (4,0) circle (1.5pt) node[right] {$a_2$};

\draw[thick, ->] (0.1,0) -- (1.9,0);
\draw[thick, ->] (3.9,0) -- (2.1,0);
\end{tikzpicture}
\end{figure}
  as an induced subgraph.
 \end{itemize}
\end{definition}

For example, consider the two mixed graphs
\begin{figure}[H]
\centering
\begin{tikzpicture}[>={Straight Barb[length=7pt,width=6pt]},thick, scale =0.75]

\draw[] (-5.8, 0) node[left] {$\Gamma_3 =~$};
\draw[fill=black] (-5,1) circle (1pt) node[left] {$a_1$};
\draw[fill=black] (-3,1) circle (1pt) node[right] {$b_1$};
\draw[fill=black] (-3,-1) circle (1pt) node[right] {$a_2$};
\draw[fill=black] (-5,-1) circle (1pt) node[left] {$b_2$};
\draw[] (-2.5, 0) node[right] {$,$};

\draw[thick, ->] (-3.1, 1) -- (-4.9, 1);
\draw[thick, ->] (-3,0.9) -- (-3,-0.9);
\draw[thick] (-4.9,-0.9) -- (-3.1,0.9);
\draw[thick, ->] (-5,-0.9) -- (-5,0.9);
\draw[thick, ->] (-4.9,-1) -- (-3.1,-1);

\draw[] (1.3, 0) node[left] {$\Gamma_4 =~$};
\draw[fill=black] (2,1) circle (1pt) node[left] {$a_1'$};
\draw[fill=black] (4,1) circle (1pt) node[right] {$b'$};
\draw[fill=black] (4,-1) circle (1pt) node[right] {$a_2'$};
\draw[fill=black] (2,-1) circle (1pt) node[left] {$a_3'$};
\draw[] (4.5, 0) node[right] {$.$};

\draw[thick, ->] (2.1, 1) -- (3.9, 1);
\draw[thick, ->] (4,-0.9) -- (4,0.9);
\draw[thick] (2,-0.9) -- (2,0.9);
\draw[thick, ->] (2.1,-0.9) -- (3.9,0.9);
\draw[thick] (2.1,-1) -- (3.9,-1);
\end{tikzpicture}
\end{figure}
They are both special, and do satisfy condition~(i) of~\autoref{def:droms mixed graphs}, but condition~(iii) is satisfied only by~$\Gamma_3$, so this is a Droms mixed graph, while~$\Gamma_4$ is not.

As it happpens for RAAGs, Droms mixed graphs characterise those T-RAAGs whose finitely generated subgroups are all T-RAAGs.

\begin{theorem}\label{thm:subgroups intro}
Let~$\Gamma$ be a mixed graph, and~$G = T(\Gamma)$ the associated T-RAAG. The following are equivalent:
 \begin{itemize}
  \item[{\rm (i)}] every finitely generated subgroup of~$G$ is again a T-RAAG;
  \item[{\rm (ii)}]~$\Gamma$ is a Droms mixed graph.
 \end{itemize}
\end{theorem}

For example,~$\Z^2$, the Klein bottle group~$K$, and~$T(\Gamma_3)$ satisfy the conditions of~\autoref{thm:subgroups intro}, while the T-RAAGs associated to the mixed graphs~$\Gamma_1,\Gamma_2,\Gamma_4$ displayed above, do not.

It is worth underlining that every T-RAAG associated to a Droms mixed graph may be constructed recursively starting from free groups and performing free products and certain semi-direct products with free abelian normal factor (see \autoref{prop: droms elementary type}), in analogy with RAAGs associated to Droms simplicial graphs.
For example,~$T(\Gamma_3)\cong \Z^2\rtimes F_2$, where~$F_2$ is the free group generated by~$a_1,a_2$.

One may relax condition~(i) of~\autoref{thm:subgroups intro} and ask which T-RAAGs have all finitely generated subgroups that are also finitely presented, specifically identifying which T-RAAGs are \emph{coherent}.
In \citep{droms1987graph}, Droms proved that a RAAG is coherent if, and only if, the associated simplicial graph is \emph{chordal}, i.e., it does not contain cycles of length greater than 3 as induced subgraph.
We prove that the same occurs for T-RAAGs.

\begin{theorem}\label{thm:coherence chordal intro}
  Let~$G = T(\Gamma)$ be a T-RAAG associated with the mixed graph~$\Gamma$.
 Then~$G$ is coherent if, and only if, the underlying simplicial graph~$\overline{\Gamma}$ is chordal.
\end{theorem}

Finally, we focus on isomorphisms of T-RAAGs.
In \citep{droms1987isomorphisms}, Droms proved that simplicial graphs are \emph{rigid}: namely, if two RAAGs are isomorphic, then the two associated simplicial graphs are equal.
This phenomenon does not occur for T-RAAGs.
For example, it is easy to see that the two mixed graphs
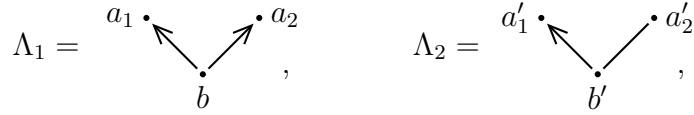
\begin{figure}[H]
\centering
\begin{tikzpicture}[>={Straight Barb[length=7pt,width=6pt]},thick, scale =0.75]

\draw[] (-5.8, 0.5) node[left] {$\Lambda_1 =~$};
\draw[fill=black] (-5,1) circle (1pt) node[left] {$a_1$};
\draw[fill=black] (-3,1) circle (1pt) node[right] {$a_2$};
\draw[fill=black] (-4,0) circle (1pt) node[below] {$b$};
\draw[] (-2.8, 0) node[right] {$,$};

\draw[thick, ->] (-4.1, 0.1) -- (-4.9, 0.9);
\draw[thick, ->] (-3.9,0.1) -- (-3.1,0.9);

\draw[] (1.3, 0.5) node[left] {$\Lambda_2 =~$};
\draw[fill=black] (2,1) circle (1pt) node[left] {$a_1'$};
\draw[fill=black] (4,1) circle (1pt) node[right] {$a'_2$};
\draw[fill=black] (3,0) circle (1pt) node[below] {$b'$};
\draw[] (4.2, 0) node[right] {$,$};

\draw[thick, ->] (2.9, 0.1) -- (2.1, 0.9);
\draw[thick] (3.1,0.1) -- (3.9,0.9);
\end{tikzpicture}
\end{figure}
define T-RAAGs which are both isomorphic to the semidirect product~$\Z\rtimes F_2$, where~$F_2$ is a free group of rank~$2$ (generated by~$a_1,a_2$, and by~$a_1',a_2'$ respectively), and the action of~$F_2$ on the normal factor is given by a non-trivial homomorphism~$F_2\to\Z^\times$.
Thus,~$\Lambda_1,\Lambda_2$ are not rigid.

To study rigidity of special mixed graphs we introduce the notion of \emph{satellite}: if~$\Gamma$ is a special mixed graph, and~$a$ is a sinkhole of~$\Gamma$, a satellite of~$a$ is a vertex~$b$ with distance 2 to~$a$ such that every vertex of~$\Gamma$ which is joined to~$b$, is joined also to~$a$.
For example, in the mixed graph~$\Lambda_1$ above the two sinkholes~$a_1,a_2$ are satellite of each other, while in the mixed graph~$\Lambda_2$ the vertex~$a_2'$ is the satellite of the sinkhole~$a_1'$.

The presence of satellites is an obstruction to rigidity.
Altogether, we prove the following.

\begin{theorem}\label{thm:isomorphism}
 Let~$T(\Gamma)$ be a T-RAAG, with associated mixed graph~$\Gamma$. Assuming that~$\Gamma$ is a special mixed graph, one has:
 \begin{itemize}
  \item[{\rm (i)}] If~$\Gamma$ is rigid then it has no satellites.
  \item[{\rm (ii)}] If~$\Gamma$ has a sinkhole which is joined to every other vertex, then~$\Gamma$ is rigid.
  \item[{\rm (iii)}] In case~$\Gamma$ is a Droms mixed graph,~$\Gamma$ is rigid if, and only if, it has not satellites. \end{itemize}
\end{theorem}

For example, the Droms mixed graph~$\Gamma_3$ displayed above is not rigid by~\autoref{thm:isomorphism}--(iii), while the special mixed graphs~$\Gamma_2$ and~$\Gamma_4$ are rigid by~\autoref{thm:isomorphism}--(ii).
We suspect that the lack of satellites is also a sufficient condition to rigidity for special mixed graphs.

\begin{conj}
 Let~$\Gamma$ be a special mixed graph.
If~$\Gamma$ has no satellites, then it is rigid.
\end{conj}

The article is structured as follows.
\Cref{Definitions_preliminaries_and_notation} introduces the notation and main definitions used throughout this article regarding mixed graphs. In \Cref{sec:TRAAGs} we define T-RAAGs and some of their properties. Next, in \Cref{sec: Coherence in T-RAAGs}, we characterize T-RAAGs that are coherent, in the spirit of \citep{droms1987graph}. \Cref{sec: When Subgroups Fail to Be T-RAAGs} serves to show that T-RAAGs defined by graphs that are not Droms mixed graphs, contain finitely generated subgroups that are not T-RAAGs. On the other hand, \Cref{sec: Elementary T-RAAGs} shows that the finitely generated subgroups of T-RAAGs build over Droms mixed graphs, are T-RAAGs themselves as well. Finally, in \Cref{sec: Rigidity of T-RAAGs} we provide a characterization of the isomorphism problem for T-RAAGs over Droms mixed graphs.

\section{Mixed Graphs
%Definitions, Preliminaries, and Notation
}\label{Definitions_preliminaries_and_notation}

This section establishes the notation and foundational concepts used throughout the article. 

Recall that a simplicial graph is a pair~$\overline{\Gamma}=(V, E)$,
where~$V$ is a set whose elements are called vertices, and~$E\subseteq \{\{x, y\}\:\mid\: x, y \in V, x \neq y\}$ is a set of paired distinct
vertices, whose elements are called edges.

\subsection{Definition of Mixed Graphs}\label{ssec:mixed graphs}

A \emph{mixed graph}~$\Gamma$ consists of an underlying simplicial graph~$\overline{\Gamma}=(V, E)$, a set of \emph{directed edges}~$D\subseteq E$,
and two maps~$o, t\colon D\to V$. We denote it as~$\Gamma = (V, E, D, o, t)$.
For a directed edge~$\mathbf{e}=\{x, y\}\in D$, the maps~$o, t$ satisfy~$o(\mathbf{e}), t(\mathbf{e}) \in\mathbf{e}$, and~$o(\mathbf{e})\neq t(\mathbf{e})$: we call these vertices respectively the \emph{origin} and the \emph{terminus} of~$\mathbf{e}$.
Moreover, if~$\mathbf{e}=\{x,y\}\in E$ is an undirected edge, i.e.~$\{x,y\}\notin D$, we denote it as~$[x,y]=[y,x]$ and we draw it as a "plain" edge, as in the mixed graph~$\mathrm{P}_2$; while if~$\{x,y\} \in D$, we denote it as~$\mathbf{e}=[o(\mathbf{e}), t(\mathbf{e})\rangle$, and we draw it as an arrow pointing at the terminus, as in the mixed graph~$\Xi_2$ (observe that in this case, either~$\mathbf{e}= [x,y\rangle$ or~$\mathbf{e}=[y,x\rangle$.

Induced subgraphs are defined in the obvious way. Moreover, we will use the following notions.

\begin{definition}\label{def:signature}\rm
Let~$\Gamma = (V, E, D, o, t)$ be a mixed graph.
\begin{itemize}
  \item[(a)] A vertex~$v \in V$ is called \emph{negative} if there exists a directed edge~$[u, v\rangle \in D$ for some~$u \in V$.
  \item[(b)] A vertex~$v \in V$ is called a \emph{sinkhole} if, for every vertex~$u \in V$ with~$\{u, v\} \in E$, it holds that~$[u, v\rangle \in D$.
  \item[(c)] A \emph{signature} of~$\Gamma$ is a map~$\theta\colon V\to\Z^\times$ satisfying the following:~$\theta(x)=-1$ for every negative vertex; and~$\theta(y)=1$ if~$\{y,z\}\in E$ for some~$z\in V$,~$z\neq y$, and~$y$ is not the terminus of any directed edge.
\end{itemize}
\end{definition}

\begin{remark}\label{rem:signature graph}
Let~$\Gamma=(V,E,D,o,t)$ be a mixed graph, and~$\theta\colon V\to\Z^\times$ a signature.
If~$x\in V$ is not an isolated vertex, then the value~$\theta(x)$ is completely determined by the incidence structure of~$\Gamma$; if~$x\in V$ is isolated -- i.e.,~$\{x,y\}\notin E$ for any vertex~$y\in V$ --, instead,~$\theta(x)$ may be arbitrarily chosen.
Therefore, a mixed graph may be endowed with a unique signature if, and only if, it has no isolated vertices.
\end{remark}

%%%%%%%%%%%%%%%%%%%%%%%%%%%%%%%%%%%%%%%%%%%
\subsection{Special mixed graphs}\label{ssec:special graphs}

We say that a mixed graph~$\Gamma=(V,E,D,o,t)$  is a \emph{special mixed graph} if the following condition is satisfied: if~$x$ is a negative vertex, and~$\{x,y\}\in E$, then~$[y,x\rangle\in D$.
In other words, every negative vertex is a sinkhole.

It is straightforward to verify that a mixed graph is special if and only if it does not contain any of the following mixed graphs on three vertices as induced subgraphs:

The three triangle-graphs:
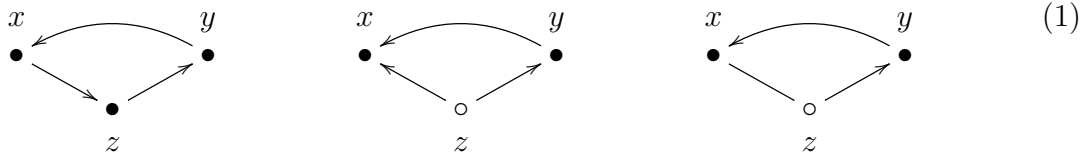
\begin{equation}\label{eq:triangle torsion proof}
  \xymatrix@R=1.5pt{x&&y \\ \bullet\ar[ddr] && \bullet\ar@/_1pc/[ll] \\ \\ &\bullet\ar[ruu]& \\ &z&}
  \qquad\qquad
  \xymatrix@R=1.5pt{x&&y \\ \bullet && \bullet\ar@/_1pc/[ll] \\ \\ &\circ\ar[uur]\ar[luu]& \\ &z&}
  \qquad\qquad
  \xymatrix@R=1.5pt{x&&y \\ \bullet && \bullet\ar@/_1pc/[ll] \\ \\&\circ\ar[ruu]\ar@{-}[uul]&\\ &z&}
\end{equation}

The two triangle-graphs:
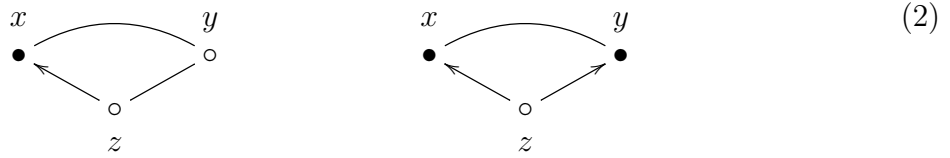
\begin{equation}\label{eq:triangle notorsion proof}
  \xymatrix@R=1.5pt{ x && y \\ \bullet && \circ\ar@{-}@/_1pc/[ll] \\ \\
  &\circ\ar[luu]\ar@{-}[ruu]& \\&z&}
  \qquad\qquad\qquad
  \xymatrix@R=1.5pt{ x && y \\ \bullet && \bullet\ar@{-}@/_1pc/[ll] \\ \\ &\circ\ar[uur]\ar[luu]& \\&z&}
\end{equation}

The two line-graphs:
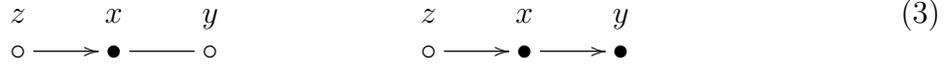
\begin{equation}\label{eq:triangle open proof}
  \xymatrix@R=1.5pt{ z&x&y \\  \circ\ar[r] &\bullet& \circ\ar@{-}[l] }
  \qquad\qquad\qquad
  \xymatrix@R=1.5pt{z&x&y \\  \circ\ar[r] &\bullet\ar[r]& \bullet }
\end{equation}

\begin{example}\label{ex:complete special}\rm
A complete mixed graph -- i.e., a mixed graph whose underlying simplicial graph is complete -- is special if, and only if, it has at most one negative vertex, which is a sinkhole.
\end{example}

Within the family of special mixed graphs we define the following subfamily.

\begin{definition}\label{def:droms}
 Let~$\Gamma = (V, E, D, o, t)$ be a mixed graph. We say that~$\Gamma$ is a \emph{Droms mixed graph} if it satisfies the following three conditions:
 \begin{itemize}
  \item[(i)]~$\Gamma$ is special;
  \item[(ii)] the underlying simplicial graph~$\overline{\Gamma}=(V,E)$ is a Droms simplicial graph -- i.e.,~$\mathrm{C}_4$ and~$\mathrm{P}_4$ are not induced subgraphs of~$\overline{\Gamma}$;
  \item[(iii)]~$\Gamma$ does not contain the mixed graph~$\Lambda_s$ as an induced subgraph.
 \end{itemize}
\end{definition}

\begin{remark}\label{T-stable subgraphs}
Induced subgraphs of Droms mixed graphs are again Droms mixed graphs.
\end{remark}

\begin{example}\label{ex:complete Droms}\rm
A complete special mixed graph is a Droms mixed graph. We underline the straightforward fact that every induced subgraph of a complete special graph is again a complete special graph.
\end{example}

In the next subsection we characterise Droms mixed graphs constructively, in analogy with Droms simplicial graphs \citep{droms1987subgroups}.

%%%%%%%%%%%%%%%%%%%%%%%%%%%%%%%%%%%%%%%%%%%%%%%%%%%%%%%%%%%5
\subsection{Droms mixed graphs}\label{ssec:droms graphs}

The disjoint union of two mixed graphs is defined in the most natural way: namely, given two mixed graphs
$$\Gamma_1 = (V_1, E_1, D_1, o_1, t_1)\qquad\text{and}\qquad\Gamma_2 = (V_2, E_2, D_2, o_2, t_2),$$ the \emph{disjoint union} of~$\Gamma_1$ and~$\Gamma_2$ is the mixed graph~$\Gamma_1\sqcup\Gamma_2 = (V, E, D, o, t)$
where~$V=V_1\sqcup V_2$,~$E=E_1\sqcup E_2$,~$D=D_1\sqcup D_2$, and~$o,t\colon D\to V$ extend~$o_1,o_2,t_1,t_2$ in the obvious way.

For mixed graphs one has also the following construction.

\begin{definition}\label{def:cone}\rm
Let~$\Gamma = (V, E, D, o, t)$ be a mixed graph, and let~$\theta\colon V \to \Z^\times$ be an orientation.
The \emph{cone}~$\nabla_\theta(\Gamma)$ of~$\Gamma$ induced by the orientation~$\theta$ is the mixed graph
\[
  \nabla(\Gamma) = \left( \tilde V,\; \tilde E,\; \tilde D,\; \tilde o,\; \tilde t \right),
\]
which contains~$\Gamma$ as an induced subgraph, defined as follows:
\begin{itemize}
  \item[(i)]~$\tilde V = V \sqcup \{w\}$, i.e., we adjoin a new vertex -- the \emph{(positive) tip of the cone};
  \item[(ii)] the set of edges is
  \[
    \tilde E = E \sqcup \left\{ \{v, w\} \mid v \in V \right\},
  \]
  meaning the tip~$w$ is joined by an edge to every vertex of~$\Gamma$.
  \item[(iii)] the set of directed edges is
  \[
    \tilde D = D \sqcup \left\{ [w, v\rangle \mid \text{$v \in V$ is negative} \right\},
  \]
  and the maps~$\tilde o, \tilde t \colon \tilde E \to \tilde V$ extend~$o$ and~$t$ accordingly.
\end{itemize}
Also, we extend the signature~$\theta\colon V\to\Z^\times$ to a signature~$\tilde\theta\colon \tilde V\to\Z^\times$ of~$\tilde\Gamma$ by putting~$\tilde\theta(w)=1$.
\end{definition}

\begin{example}\label{ex:nospecial nocone}\rm
The mixed graphs in \eqref{eq:triangle torsion proof}--\eqref{eq:triangle open proof} are not cones.
\end{example}

\begin{example}\label{ex:Lambda special}\rm
The mixed graph~$\Lambda_{\mathrm s}$ given by
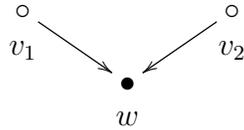
\[
 \xymatrix@R=1.5pt{ \circ\ar[ddr] && \circ\ar[ddl] \\ v_1&&v_2\\ &\bullet& \\&w&}
\]
is special — the only negative vertex is~$w$, which is a sinkhole — but it is not a cone.
\end{example}

\begin{example}\label{ex:P_4 - C_4}\rm
If~$\Gamma$ is a mixed graph such that~$\overline{\Gamma} = \mathrm{P}_4$ or~$\overline{\Gamma} = \mathrm{C}_4$, then~$\Gamma$ cannot be a cone.
\end{example}

\begin{example}\label{ex:cone complete}\rm
Let~$\Gamma=(V,E,D,o,t)$ be a complete special mixed graph with positive vertices~$v_1,\ldots,v_d$,~$d\geq1$, and let~$\theta\colon V\to\Z^\times$ the only signature of~$\Gamma$, that is,~$\theta(v_i)=1$ for all~$i=1,\ldots,d$, and ~$\theta(w)=-1$ for the only sinkhole~$w$, if there is one.
For every~$i=1,\ldots,d$, let~$\Gamma_i$ be the induced subgraph of~$\Gamma$ with vertex set~$V_i=V\smallsetminus\{v_i\}$.
Then
\[
 \Gamma=\nabla_{\theta\vert_{V_1}}(\Gamma_1)=\ldots=\nabla_{\theta\vert_{V_d}}(\Gamma_d),
\]
where the tip of each cone is, respectively,~$v_1,\ldots,v_d$.
\end{example}

Droms mixed graphs, defined in the previous subsection, may be built recursively as shown by the following, see also \cite*[Prop.~2.14]{blumer2024oriented}.

\begin{proposition}\label{prop: droms elementary type}
A mixed graph~$\Gamma = (V, E, D, o, t)$ is a Droms mixed graph if, and only if, the following holds: for every induced subgraph~$\Lambda$ of~$\Gamma$, either
\begin{itemize}
  \item[(a)]~$\Lambda$ is the disjoint union of two non-empty proper subgraphs; or
  \item[(b)]~$\Lambda$ is the cone~$\nabla_{\theta_0}(\Lambda_0)$ for some induced subgraph~$\Lambda_0 = (V_0, E_0, D_0, o\vert_{E_0}, t\vert_{E_0})$ and some signature~$\theta_0\colon V_0\to\Z^\times$ of~$\Lambda_0$.
\end{itemize}
\end{proposition}

\begin{proof}
We first show that any Droms mixed graph~$\Gamma=(V, E, D, o, t)$ satisfies conditions~(a)--(b), using induction on the number of vertices.
It is straightforward to verify that any mixed graph with at most two vertices is a Droms mixed graph and satisfies both conditions.

Assume now that~$\Gamma$ has at least three vertices.
Since the underlying simplicial graph~$\bar\Gamma$ is a Droms simplicial graph, by \cite[Lemma, p.~520]{droms1987subgroups} either~$\Gamma$ is disconnected -- in which case it satisfies condition~(a) -- or there exists a vertex~$c \in V$ adjacent to every other vertex, that is,~$\{c,v\} \in E$ for all~$v \in V \smallsetminus \{c\}$.

If there is such a vertex~$c$, we claim that either~$\Gamma$ is a complete special mixed graph (cf.~\autoref{ex:complete special}); or it is the cone~$\nabla_{\theta_0}(\Gamma_0)$, with tip~$c$,~$\Gamma_0$ the induced subgraph on the vertex set~$V_0 = V \smallsetminus \{c\}$, and~$\theta_0$ a signature of~$\Gamma_0$.

Suppose first that~$[u,c\rangle \in D$ for some~$u \in V$.
Then~$c$ is a sinkhole, since~$\Gamma$ is a Droms mixed graph (and thus, special), and hence~$[v,c\rangle \in D$ for every~$v \in V \smallsetminus \{c\}$.
Furthermore, for any pair of distinct vertices~$v_1, v_2 \in V \smallsetminus \{c\}$, we must have~$\{v_1, v_2\} \in E$, as otherwise the induced subgraph on~$\{v_1, v_2, c\}$ would be isomorphic to the mixed graph~$\Lambda_{\mathrm s}$, which is prohibited by condition~(iii) of~\autoref{def:droms}.
Moreover, if~$[v_1, v_2\rangle \in D$, then~$v_2$ would be a negative vertex that is not a sinkhole, since~$[v_2, c\rangle \in D$, contradicting the assumption that~$\Gamma$ is special.
Hence, in this case,~$\Gamma$ is a complete special mixed graph, which is a cone (cf. \autoref{ex:cone complete}).

On the other hand, if~$\{v,c\}\in E$ and~$[v, c\rangle \notin D$ for every~$v \in V\smallsetminus\{c\}$, consider the induced subgraph~$\Gamma_0 = (V_0, E_0, D_0,o\vert_{D_0},t\vert_{D_0})$ with vertex set~$V_0 = V \smallsetminus \{c\}$.
For every~$v\in V_0$ we set
\[      \theta_0(v) =\begin{cases} 1 &\text{if }[v,c]\in E,\\
-1 &\text{if }[c,v\rangle\in D.                     \end{cases}  \]
We claim that the map~$\theta_0\colon V_0\to\Z^\times$ is a signature, and~$\Gamma=\nabla_{\theta_0}(\Gamma_0)$, with positive tip~$c$.
\begin{itemize}
 \item[(1)] If~$v \in V_0$ is a negative vertex of~$\Gamma_0$, then it is a sinkhole -- as~$\Gamma$ is special, so that~$[c, v\rangle \in D$.
Hence,~$\theta_0(v)=-1$.
 \item[(2)] If~$[v,c]\in E$ -- and hence~$\theta_0(v)=1$ --, then~$[v',v\rangle\notin D$ for any~$v'\in V_0$,~$v'\neq v$, as otherwise~$v$ would be a negative vertex of~$\Gamma$ -- a special mixed graph -- but not a sinkhole.
\end{itemize}
Altogether,~$\theta_0\colon V_0\to \Z^\times$ is a signature of~$\Gamma_0$, and~$\Gamma=\nabla_{\theta_0}(\Gamma_0)$ with tip~$c$, as claimed.
This completes the proof that a Droms mixed graph satisfies conditions~(a)--(b).

Conversely, suppose that~$\Gamma=(V, E, D, o, t)$ satisfies both conditions~(a)--(b).
\begin{itemize}
 \item[(1)] If~$\Gamma$ contains~$\Lambda_{\mathrm s}$ as an induced subgraph, then~$\Gamma$ does not satisfy the conditions~(a)--(b), as~$\Lambda_{\mathrm s}$ is connected and is not a cone (cf.~\autoref{ex:Lambda special}) — a contradiction.
 \item[(2)] Similarly, if~$\Gamma$ contains an induced subgraph~$\Lambda$ such that~$\bar\Lambda = \mathrm{P}_4$ or~$\bar\Lambda = \mathrm{C}_4$, then~$\Lambda$ is not a cone (cf.~\autoref{ex:P_4 - C_4}), again contradicting conditions~(a)--(b).
\end{itemize}
Finally, suppose that~$\Gamma$ is not special. Then it contains at least one of the mixed graphs with three vertices described in \eqref{eq:triangle torsion proof}--\eqref{eq:triangle open proof} as an induced subgraph; each of these graphs is connected and not a cone, which again contradicts the conditions~(a)--(b).
Therefore,~$\Gamma$ is a Droms mixed graph.
\end{proof}

\begin{example}\label{ex:gamma3 droms}
 The mixed graph~$\Gamma_3$ displayed in the Introduction is a Droms mixed graph. It may be constructed starting from the two mixed graphs consisting of the single vertex~$a_1$ and the single vertex~$a_2$, as displayed here below:
 
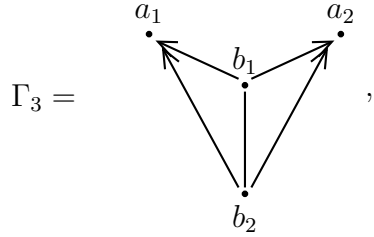
\begin{figure}[H]
\centering
\begin{tikzpicture}[>={Straight Barb[length=7pt,width=6pt]},thick, scale =0.85]

\draw[] (-0.8, 0) node[left] {$\Gamma_3 =~$};
\draw[fill=black] (0,1) circle (1pt) node[above] {$a_1$};
\draw[fill=black] (3,1) circle (1pt) node[above] {$a_2$};
\draw[fill=black] (1.5,0.2) circle (1pt) node[above] {$b_1$};
\draw[fill=black] (1.5,-1.5) circle (1pt) node[below] {$b_2$};
\draw[] (3.2, 0) node[right] {$,$};

\draw[thick, ->] (1.4, 0.3) -- (0.1, 0.9);
\draw[thick, ->] (1.6,0.3) -- (2.9,0.9);
\draw[thick, ->] (1.4, -1.4) -- (0.2, 0.8);
\draw[thick, ->] (1.6,-1.4) -- (2.8,0.8);
\draw[thick] (1.5,-1.4) -- (1.5,0.1);
\end{tikzpicture}
\end{figure}
namely,~$  \Gamma_3=\nabla_{\tilde\theta}(\nabla_\theta(\{a_1\}\sqcup\{a_2\}))$
where~$a_1,a_2$ are negative vertices, and~$b_1,b_2$ are the two tips.
\end{example}

\begin{example}\label{ex:upsilon droms}
 The mixed graph
 
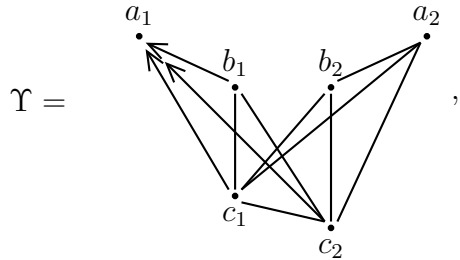
\begin{figure}[H]
\centering
\begin{tikzpicture}[>={Straight Barb[length=7pt,width=6pt]},thick, scale =0.85]

\draw[] (-0.8, 0) node[left] {$\Upsilon =~$};
\draw[fill=black] (0,1) circle (1pt) node[above] {$a_1$};
\draw[fill=black] (4.5,1) circle (1pt) node[above] {$a_2$};
\draw[fill=black] (1.5,0.2) circle (1pt) node[above] {$b_1$};
\draw[fill=black] (3,0.2) circle (1pt) node[above] {$b_2$};
\draw[fill=black] (1.5,-1.5) circle (1pt) node[below] {$c_1$};
\draw[fill=black] (3,-2) circle (1pt) node[below] {$c_2$};
\draw[] (4.7, 0) node[right] {$,$};

\draw[thick, ->] (1.4, 0.3) -- (0.1, 0.9);
\draw[thick] (3.1,0.3) -- (4.4,0.9);
\draw[thick, ->] (1.4, -1.4) -- (0.1, 0.8);
\draw[thick] (1.6,-1.4) -- (4.4,0.9);
\draw[thick] (1.5,-1.4) -- (1.5,0.1);
\draw[thick] (1.6, -1.4) -- (2.9, 0.1);

\draw[thick] (3.1,-1.9) -- (4.4,0.9);
\draw[thick] (3,-1.9) -- (3,0.1);
\draw[thick] (2.9,-1.9) -- (1.6,0.1);
\draw[thick] (2.9,-1.9) -- (1.6,-1.6);
\draw[thick, ->] (2.9,-1.9) -- (0.4,0.6);

\end{tikzpicture}
\end{figure}
is a Droms mixed graph, and it may be obtained starting from the two mixed graphs consisting of the single negative vertex~$a_1$ and the single positive vertex~$a_2$: namely,
$$  \Upsilon=\nabla_{\tilde\theta}\left(\nabla_\theta
\left(\nabla_{\theta_1}(\{a_1\})\sqcup\nabla_{\theta_2}(\{a_2\})\right)\right),$$
with~$\theta_1(a_1)=-1$ and~$\theta_2(a_2)=1$.
\end{example}

%%%%%%%%%%%%%%%%%%%%%%%%%%%%%%%%%%%%%%%%%%%%%%%%%%%%%%%%%%%%%5
%%%%%%%%%%%%%%%%5
%%%%%%%%%%%%%%%%%%%%%%%%%%%%%%%%%%%%%%%%%%%%%%%%%%%%%%%%%%%%%%5
%%%%%%%%%%%%%%%%%%%
%%%%%%%%%%%%%%%%%%%%%%%%%%%%%%%%%%%%%%%%%%%%%%%%%%%%%%%%%%%%%%

\section{Twisted right-angled Artin groups}\label{sec:TRAAGs}

%%%%%%%%%%%%%%%%%5

All groups considered will be given by finite presentations of the form~$G = \gp{A}{R}$, where an element~$g \in G$ may be expressed as a word in~$(A \cup A^{-1})^{\ast}$. If~$X$ is a subset of~$G$, then~$\sgp{X}$ denotes the subgroup of~$G$ generated by~$X$.

For a finite simplicial graph~$\overline{\Gamma} = (V, E)$, we denote the associated right-angled Artin group by~$A(\Gamma)$, defined by the presentation:
\[
A(\overline{\Gamma}) \coloneqq \langle\, V \mid uv = vu \text{ whenever } \{u,v\}\in E\,\rangle.
\]

\subsection{Definition and properties of T-RAAGs}
When interpreting group presentations from mixed graphs, we associate~$[a,b]$ with the commutation relation~$ab = ba$, and~$[a,b\rangle$ with the Klein relation~$aba = b$. By slight abuse of notation, we write~$[a,b] = aba^{-1}b^{-1}$ and~$[a,b\rangle = abab^{-1}$.

\begin{definition}\label{def_of_traags}
Let~$\Gamma = (V,E,D,o,t)$ be a mixed graph.
The \emph{twisted right-angled Artin group} (T-RAAG) based on~$\Gamma$ is the group
\[
T(\Gamma) = \langle\: V \:\mid\: ab = ba \text{ if } [a,b] \in E\smallsetminus D, \;\; aba = b \text{ if } [a,b\rangle \in D \:\rangle.
\]
We and refer to~$\Gamma$ as the \emph{defining graph} of~$T(\Gamma)$.
\end{definition}

If~$D = \emptyset$, then~$T(\Gamma)$ reduces to the usual RAAG~$A(\Gamma)$, where we identify the mixed graph~$\Gamma$ (without directed edges) with its underlying simplicial graph~$\overline{\Gamma}$.

\begin{remark}\label{rem: underlying RAAG of a T-RAAG}
Given a mixed graph~$\Gamma = (V, E, D, o, t)$, let~$\overline{\Gamma}$ denote its underlying simplicial graph. The group~$A(\overline{\Gamma})$ is then referred to as the \emph{underlying RAAG} of~$T(\Gamma)$. A key theme is to explore connections between~$T(\Gamma)$ and~$A(\overline{\Gamma})$.
\end{remark}

The following properties of T-RAAGs follow from normal forms in T-RAAGs \citep{foniqi2024twisted,antolin2025traag}.

\begin{proposition}\label{proposition: subgraph injectivity}

Let~$\Gamma=(V,E,D,o,t)$ be a  mixed graph, and consider the associated T-RAAG~$T(\Gamma)$.
\begin{itemize}
 \item[(i)] If~$\Delta=(V_\Delta,E_\Delta,D_\Delta, o\vert_{D_{\Delta}},t\vert_{D_{\Delta}})$ is an induced subgraph of~$\Gamma$, then the morphism
\[
i: T(\Delta) \longrightarrow T(\Gamma) \quad \text{induced by} \quad V_\Delta\hookrightarrow V,
\]
is injective. In particular,~$T(\Delta)$ is a subgroup of~$T(\Gamma)$.
 \item[(ii)]
The map~$q: A(\overline{\Gamma}) \longrightarrow T(\Gamma)$ defined by~$v \mapsto v^2$ is injective.
 \item[(iii)]
Let~$u \in V$, and set~$\Gamma_u$ to be the graph  obtained by~$\Gamma$, by changing the edges~$[v, u\rangle$ to~$[v, u]$. The map~$\varphi_u: T(\Gamma_u) \longrightarrow T(\Gamma)$ defined by~$u \mapsto u^2$, and~$v \mapsto v$ for~$v \neq u$, is injective.
\end{itemize}
\end{proposition}

\begin{cor}\label{cor: A(P_4) and A(C_4) injectivity}
If~$\overline{\Gamma}$ is not a Droms simplicial graph -- i.e., it contains~$\mathrm{P}_4$ or~$\mathrm{C}_4$ as an induced subgraph --, then the T-RAAG~$T(\Gamma)$ contains the RAAG~$A(\mathrm{P}_4)$ or~$A(\mathrm{C}_4)$, respectively.
\end{cor}

\begin{remark}\label{rem:no of tails}
Let~$\Gamma=(V,E,D,o,t)$ be a mixed graph; consider the abelianisation~$T(\Gamma)^{ab}$ of the associated T-RAAG~$T(\Gamma)$.
Then~$T(\Gamma)^{ab}$ is the direct product of a finitely generated free abelian group with a 2-elementary abelian group.
In particular,
\[
 T(\Gamma)^{ab}\simeq \Z^r\times(\Z/2\Z)^s,
\]
where~$s=|o(D)|$ -- namely,~$s$ is the number of vertices which are origin of directed edges.
\end{remark}

Let~$\Gamma=(V,E,D,o,t)$ be a mixed graph, and let~$\theta\colon V\to\Z^\times$ be a signature.
Then~$\theta$ extends to a homomorphism of groups~$T(\Gamma)\to\Z^\times$, as
\[
 [\theta(x),\theta(y)]=\theta(z)^2=1\qquad\text{for every }x,y,z\in T(\Gamma).
\]
We call the induced homomorphism~$T(\Gamma)\to\Z^\times$ a signature of~$T(\Gamma)$, and with an abuse of notation we denote it by~$\theta$ too.

%%%%%%%%%%%%%%%%%%%%%%%%%%%%%%%%%%%%%%%%%%%%%%%%%%%%%%%%%%%%%5
%%%%%%%%%%%%%%%%5
%%%%%%%%%%%%%%%%%%%%%%%%%%%%%%%%%%%%%%%%%%%%%%%%%%%%%%%%%%%%%%5

\subsection{Operations}\label{ssec:operations}

Let~$\Gamma_1$ and~$\Gamma_2$ be two mixed graphs, and let~$\Gamma=\Gamma_1\sqcup\Gamma_2$ denote their disjoint union.
As it happens for RAAGs, the T-RAAG~$T(\Gamma)$ associated to~$\Gamma$ is the free product of the two T-RAAGs~$T(\Gamma_1)$ and~$T(\Gamma_2)$.
If
\[
 \theta_1\colon T(\Gamma_1)\longrightarrow\Z^\times\qquad\text{and}\qquad
 \theta_2\colon T(\Gamma_2)\longrightarrow\Z^\times
\]
are two signatures, then the universal property of free products yields a signature~$\theta\colon T(\Gamma)\to\Z^\times$, which is the one extending~$\theta_1$ and~$\theta_2$.

Now let~$\Gamma$ be a mixed graph, and let~$\theta\colon T(\Gamma)\to\Z^\times$ be a signature.
The T-RAAG~$T(\nabla_\theta(\Gamma))$ based on the cone of~$\Gamma$ with signature~$\theta$ decomposes as semi-direct product
\[
 T(\nabla_\theta(\Gamma))=\langle\:w\:\rangle\rtimes_{\theta}T(\Gamma),
\]
where~$w$ is the tip of the cone,~$\langle\:w\:\rangle\simeq\Z$, and one has
\[
 x^{-1}wx=w^{\theta(x)}\qquad\text{for every }x\in T(\Gamma).
\]

T-RAAGs which may be constructed recursively performing free products and semi-direct products as above starting from T-RAAGs based on isolated vertices, endowed with some signatures, are called \emph{elementary T-RAAGs}.

As a consequence of~\autoref{prop: droms elementary type} one has the following.

\begin{cor}\label{cor:elementary}
A mixed graph~$\Gamma$ is a Droms mixed graph if, and only if, the associated T-RAAG~$T(\Gamma)$ is an elementary T-RAAG.
\end{cor}

%%%%%%%%%%%%%%%%%%%%%%%%%%%%%%%%%%%%%%%%%%%%%%%%%%%%%%%%%%%%%5
%%%%%%%%%%%%%%%%5
%%%%%%%%%%%%%%%%%%%%%%%%%%%%%%%%%%%%%%%%%%%%%%%%%%%%%%%%%%%%%%5

\subsection{T-RAAGs and special mixed Graphs}\label{sec:special TRAAGs}

T-RAAGs associated to special mixed graphs are characterised as follows.

\begin{proposition}\label{proposition:special}
 Let~$T(\Gamma)$ be a T-RAAG.
 The defining mixed graph~$\Gamma=(V,E,D,o,t)$ is special if, and only if, there exists a complete special mixed graph~$\Delta=(V_\Delta,E_\Delta,D_\Delta,o_\Delta,t_\Delta)$ with the same number of vertices, endowed with an epimorphism~$T(\Gamma)\to T(\Delta)$.
\end{proposition}

\begin{proof}
Suppose first that there is an epimorphism~$\phi\colon T(\Gamma)\to T(\Delta)$ with~$\Delta$ a complete special mixed graph with the same vertices as~$\Gamma$.
By~\autoref{proposition: subgraph injectivity}, for every induced subgraph~$\Gamma'$ of~$\Gamma$,~$\phi$ restricts to an epimorphism~$\phi\vert_{T(\Gamma')}\colon T(\Gamma')\to T(\Delta')$, where~$\Delta'$ is an induced subgraph of~$\Delta$ -- hence,~$\Delta'$ is again a complete special mixed graph, cf.~\autoref{ex:complete Droms} -- with the same vertices as~$\Gamma'$.
In particular, if~$\Gamma'$ has three vertices, then~$T(\Gamma')$ has epimorphic image~$\Z^3$ (when~$\Delta'$ is a triangle without sinkholes), or~$\Z^2\rtimes \Z$ (when~$\Delta'$ is a triangle with a sinkhole).
If~$\Gamma'$ is like one of the mixed graphs with three vertices displayed in \eqref{eq:triangle torsion proof}--\eqref{eq:triangle open proof}, then~$T(\Gamma')$ cannot have such an epimorphic image.
Then~$\Gamma$ has no induced subgraphs like those displayed in \eqref{eq:triangle torsion proof}--\eqref{eq:triangle open proof}, and thus~$\Gamma$ is special.

Conversely, assume that~$\Gamma$ is special.
Without loss of generality, we may assume that~$T(\Gamma)$ does not decompose as a non-trivial free product -- namely,~$\Gamma$ is connected.
In particular,~$\Gamma$ admits a unique signature~$\theta\colon T(\Gamma)\to\Z^\times$ (cf.~\autoref{rem:signature graph}).
Now let~$x_1,\ldots,x_r\in V$ be the negative vertices of~$\Gamma$, and let~$y_1,\ldots,y_s\in V$ be the positive vertices. Since~$\Gamma$ is special, the negative vertices are sinkholes, and thus they are pairwise disjoint.

If~$\Gamma$ has no sinkholes, then~$\Gamma$ has no directed edges and thus it is a simplicial graph,~$T(\Gamma)\simeq A(\Gamma)$, and hence~$T(\Gamma)^{ab}\simeq\Z^s\simeq T(\Delta)$, where~$\Delta$ is the complete simplicial graph with~$s$ vertices.

If~$\Gamma$ has at least a sinkhole, let~$N$ be the normal subgroup of~$T(\Gamma)$ generated as a normal subgroup by the following elements:
\[
 [y_i,y_{i'}],\qquad y_i^{x_1}y_i,\qquad x_jx_1^2x_j,\qquad\text{with }1\leq i,i'\leq s,\:1<j\leq r.
\]
Then the quotient~$T(\Gamma)/N$ is isomorphic to the T-RAAG~$T(\Delta)$ based on the complete mixed graph~$\Delta=(V_\Delta,E_\Delta,D_\Delta,o_\Delta,t_\Delta)$, with
\[
 \begin{split}
  V_{\Delta}&=\left\{\:x_1,\:x_1x_2,\ldots,x_1x_r,\:y_1,\:\ldots,\:y_s\:\right\},\\
  E_\Delta&=\left\{\:[y_i,y_i'],\:[y_i,x_1x_j],\:[y_i,x_1\rangle,\:[x_1x_j,x_y\rangle\:\mid\:1\leq i,i'\leq s,\:1<j\leq r\:\right\}.
 \end{split}
\]
The complete mixed graph~$\Delta$ is special as it has a unique negative vertex,~$x_1$, which is a sinkhole.
\end{proof}

%%%%%%%%%%%%%%%%%%%%%%%%%%%%%%%%%%%%%%%%%%%%%%%%%%%%%%%%%%%%%%%%%5
%%%%%%%%%%%%5
%%%%%%%%%%%%%%%%%%%%%%%%%%%%%%%%%%%%%%%%%55
%%%%%%%%%55
%%%%%%%%%%%%%%%%%%%%%%%%%%%%%%%%%%%%%%%%%%%%%%%%%%%%%%%%%%%%%%%%%%%%%5

\section{Coherence in T-RAAGs}\label{sec: Coherence in T-RAAGs}

A group is \emph{coherent} if all of its finitely generated subgroups are finitely presented. This work provides an explicit characterization of the coherence of T-RAAGs, emphasizing their similarity to RAAGs~\citep{droms1987graph}.

A (simplicial) graph~$\Gamma$ is called \emph{chordal} if all cycles of four or more vertices have a \emph{chord}, which is an edge that is not part of the cycle but connects two vertices of the cycle. Droms provided the following simple characterisation of coherent RAAGs. 

\begin{theorem}[{\citealp[Theorem 1]{droms1987graph}}]
Let~$\Gamma$ be a finite simplicial graph. Then~$A(\Gamma)$ is coherent if and only if~$\Gamma$ is chordal.
\end{theorem}

\begin{lemma}[{\citealp[Theorem 8]{karrass1970subgroups}}]\label{lemma: facts coherent}
Let~$A, B$ be groups and~$C$ a common subgroup. If~$A$ and~$B$ are coherent and all subgroups of~$C$ are finitely generated, then the group~$G = A *_C B$ is coherent.
\end{lemma}

For a finite mixed graph~$\Gamma$ we will say that~$\Gamma$ is chordal if~$\overline{\Gamma}$ is chordal.

\begin{theorem}\label{thm: coherence}
The T-RAAG~$T(\Gamma)$ is coherent if and only if~$\Gamma$ is chordal.
\end{theorem}

\begin{proof}
If~$\Gamma$ is not chordal, then~$A(\overline{\Gamma})$ is not coherent. From \autoref{proposition: subgraph injectivity} we know that~$A(\overline{\Gamma})$ injects into~$T(\Gamma)$, hence~$T(\Gamma)$ is not coherent.

On the other hand, assume~$\Gamma$ is chordal. In the case when~$\Gamma$ is complete, the group~$T(\Gamma)$ is virtually free abelian, and hence coherent as well. If~$\Gamma$ is not complete, then there are two proper subgraphs~$\Gamma_1$ and~$\Gamma_2$ with~$\Gamma = \Gamma_1 \cup \Gamma_2$ and~$\Gamma_1 \cap \Gamma_2 = C$, with~$C$ complete. Here one has a splitting:
\[
T(\Gamma) = T(\Gamma_1) *_ {T(C)} T(\Gamma_2).
\]
As~$C$ is complete,~$T(C)$ is virtually free abelian. Now the result follows from \autoref{lemma: facts coherent}.
\end{proof}

\section{When Subgroups Fail to Be T-RAAGs}\label{sec: When Subgroups Fail to Be T-RAAGs}

In this section, we show that the excluded graphs correspond to T-RAAGs that contain finitely generated subgroups which are not themselves T-RAAGs.

\subsection{Poisonous subgraphs}

We begin with the cases of~$P_4$ and~$C_4$. These cases follow the argument in \citep{droms1987subgroups}. Start with a mixed graph~$\Gamma$, and assume that~$\overline{\Gamma}$ contains~$P_4$, or~$C_4$. Then, by \autoref{proposition: subgraph injectivity}, our group~$T(\Gamma)$ contains~$A(P_4)$ or~$A(C_4)$. So, it suffices to show that~$A(P_4)$ and~$A(C_4)$ contain finitely generated subgroups that are not T-RAAGs.

The group~$A(C_4)$ is isomorphic to~$F_2 \times F_2$. This group is not coherent; that is, it contains a finitely generated subgroup~$H$ that is not finitely presented. Clearly, such a subgroup~$H$ cannot be a twisted graph group.

For the case of 
\[
A(P_4) = \langle x, y, z, w \mid xy = yx,\ yz = zy,\ zw = wz \rangle,
\]
consider the homomorphism~$\varphi \colon A(P_4) \to \mathbb{Z}_2 = \langle t \rangle$ that maps each generator of~$A(P_4)$ to~$t$. As shown in \citep{droms1987subgroups}, using the Reidemeister-Schreier method, the kernel of~$\varphi$ admits the presentation:
\[
H = \langle a, b, c, d \mid ab = ba,\ bc = cb,\ bc^2d = dbc^2 \rangle,
\]
where~$a, b, c, d$ correspond to the elements~$xy^{-1}, y^2, y^{-1}z, z^{-1}w$, respectively. \cite{droms1987subgroups} showed that~$H$ is not a graph group. Here, we show that~$H$ is not a twisted graph group either.

Suppose for contradiction that~$H$ is a twisted graph group, say~$H = T(\Gamma)$ for some mixed graph~$\Gamma$. Since~$K^{ab} \cong \mathbb{Z}^4$, we know that~$\Gamma$ must have exactly~$4$ vertices, and no directed edges; this follows from \cite[Subsection 5.1]{foniqi2024twisted}, which states:
\[
T(\Gamma)^{ab} \cong \mathbb{Z}^{|V\Gamma - V_o|} \times \mathbb{Z}_2^{|V_o|},
\]
where~$V_o$ denotes the set of vertices of~$\Gamma$ that are the origin of a directed edge. 

Consequently, we have shown that if~$H$ is a T-RAAG, it must actually be a RAAG. However, \cite{droms1987subgroups} demonstrated that~$H$ cannot be a RAAG. Therefore, we have proved that~$A(P_4)$ contains finitely generated subgroups that are not T-RAAGs.

\subsection{Non-special mixed graphs}

Next, we consider the graphs from \Cref{ssec:special graphs} and \autoref{ex:Lambda special}.  
In fact, it suffices to prove the result for the following three graphs, which we denote by~$\Delta$,~$L$, and~$\Lambda_{\mathrm{s}}$, respectively:
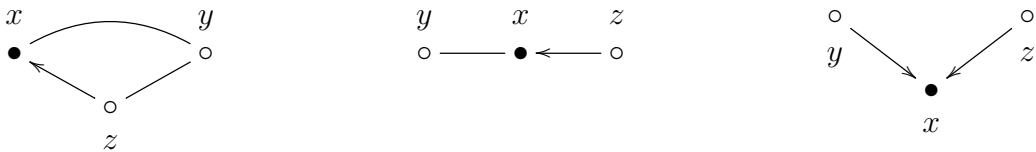
\begin{equation*}
  \xymatrix@R=1.5pt{ 
    x && y \\ 
    \bullet && \circ\ar@{-}@/_1pc/[ll] \\ \\
    &\circ\ar[luu]\ar@{-}[ruu]& \\ 
    &z&
  }
  \qquad\qquad\qquad
  \xymatrix@R=1.5pt{ 
    y & x & z \\  
    \circ\ar@{-}[r] & \bullet & \circ\ar[l] 
  }
  \qquad\qquad\qquad
  \xymatrix@R=1.5pt{ 
    \circ\ar[ddr] && \circ\ar[ddl] \\ 
    y && z\\ 
    &\bullet& \\ 
    &x&
  }
\end{equation*}

Indeed, all T-RAAGs based on the triangle graphs in \eqref{eq:triangle torsion proof}--\eqref{eq:triangle notorsion proof} contain the T-RAAG~$T(\Delta)$, by appropriately squaring some of the vertices in those graphs, via \autoref{proposition: subgraph injectivity}.  
Likewise, an analogous argument applies to the second line graph in \eqref{eq:triangle open proof}, whose corresponding T-RAAG contains~$T(L)$.

Now, in the T-RAAG
\[
T(\Delta) = \langle x, y, z \mid xy = yx,\ yz = zy,\ zxz = x \rangle,
\]
consider the subgroup~$H = \langle x, t \rangle$, where~$t = yz$. If~$H$ is a T-RAAG~$T(\Gamma)$ for some mixed graph~$\Gamma$, then~$|V\Gamma| = 2$. 

Since~$T(\Delta)$ is commensurable with~$\Z^3$ (the subgroup~$N = \langle x^2, y, z \rangle$ is normal of index~$2$ and is isomorphic to~$\Z^3$), it follows that~$H$ must be commensurable with~$\Z^2$. Thus,~$\Gamma$ must be either an undirected edge or a directed one; that is,~$T(\Gamma) \cong \Z^2$ or~$T(\Gamma) \cong K$, where~$K$ denotes the fundamental group of the Klein bottle.

Consider~$t^x = x t x^{-1}$. Substituting~$t = yz$, we compute:
\[
t^x = x y z x^{-1} = y x z x^{-1} = y z^{-1}.
\]
Hence,~$x,\ yz,\ yz^{-1} \in H$, which implies that~$\langle x^2, y^2, z^2 \rangle \cong \Z^3$ is a subgroup of~$H = T(\Gamma)$. This contradicts the assumption that~$T(\Gamma) \cong \Z^2$ or~$T(\Gamma) \cong K$.

Next, in the T-RAAG
\[
T(L) = \langle x, y, z \mid xy = yx,\ zxz = x \rangle,
\]
consider (as above) the subgroup~$H = \langle x, t \rangle$, where~$t = yz$. If~$H$ is a T-RAAG~$T(\Gamma)$ for some mixed graph~$\Gamma$, then~$|V\Gamma| = 2$.

As above, we compute
\[
t^x = x t x^{-1} = x y z x^{-1} = y x z x^{-1} = y z^{-1}.
\]
Hence,~$x,\ yz,\ yz^{-1} \in H$, which implies that~$\langle x^2, y^2, z^2 \rangle \cong \Z \times F_2$ is a subgroup of~$H = T(\Gamma)$. 

In particular,~$\Gamma$ cannot consist of two disjoint vertices, as this would make~$H$ a free group, and hence could not contain a subgroup isomorphic to~$\Z \times F_2$. Thus, either~$T(\Gamma) \cong \Z^2$ or~$T(\Gamma) \cong K$; but neither of these groups contains free subgroups, and therefore cannot contain a subgroup isomorphic to~$\Z \times F_2$.

\subsection{The mixed graph~$\Lambda_{\mathrm s}$}

Finally, in the T-RAAG
\[
T(\Lambda_{\mathrm{s}}) = \langle x, y, z \mid yxy = x,\ zxz = x \rangle,
\]
consider the subgroup~$H = \langle x, t \rangle$, where~$t = yz$. If~$H$ is a T-RAAG~$T(\Gamma)$ for some mixed graph~$\Gamma$, then~$|V\Gamma| = 2$. Note that the subgroup~$\langle y, z \rangle$ is a free group of rank~$2$. 

As above, we compute:
\[
t^x = x t x^{-1} = x y z x^{-1} = y^{-1} z^{-1}.
\]
Hence,~$yz,\ y^{-1} z^{-1} \in H$. Since~$\langle y, z \rangle$ is free, the subgroup~$\langle yz, y^{-1} z^{-1} \rangle$ is also free, which implies that~$\Gamma$ must consist of two disjoint vertices.

However, this leads to a contradiction, since~$x^2, t \in H$, and~$\langle x^2, t \rangle \cong \Z^2$, which cannot occur in a free group. 

Consequently,~$T(\Lambda_{\mathrm{s}})$ contains a finitely generated subgroup that is not a T-RAAG.

This completes the argument that each of the graphs identified above gives rise to T-RAAGs containing finitely generated subgroups which are not themselves T-RAAGs.

\section{Elementary T-RAAGs}\label{sec: Elementary T-RAAGs}

In this section we prove the implication (ii)$\Rightarrow$(i) of~\autoref{thm:subgroups intro}, using the recursive description of T-RAAGs based on Droms mixed graphs given by~\autoref{cor:elementary}.

In fact, we prove the following slightly stronger result.

\begin{proposition}\label{prop:droms subgroups}
 Let~$\Gamma$ be a Droms mixed graph, and let~$\theta\colon T(\Gamma)\to\Z^\times$ be a signature.
 Then every finitely generated subgroup~$H$ of~$T(\Gamma)$ is again a T-RAAG, and moreover the restriction~$\theta\vert_H\colon H\to\Z^\times$ is a signature of the T-RAAG~$H$.
\end{proposition}

First, we observe that for the "building block" of elementary T-RAAGs -- namely,~$\Z$ considered as the T-RAAG based on a single vertex --~\autoref{prop:droms subgroups} holds trivially.

Then, we show that the property prescribed by~\autoref{prop:droms subgroups} is preserved by the two operations involved in the construction of elementary T-RAAGs: free products, and semi-direct products.

%%%%%%%%%%%%%%%%%%%%%%%%%%%%%%%%%%%%%%%%%%

\subsection{Free products}

If~$\Gamma = \Gamma_1 \sqcup \Gamma_2$, then~$T(\Gamma) = T(\Gamma_1) \ast T(\Gamma_2)$, and the following result follows from the Kurosh Subgroup Theorem, and from the universal property of free products.

\begin{proposition}\label{prop:free prod}
Let~$\Gamma = \Gamma_1 \sqcup \Gamma_2$ be a disconnected mixed graph, and let~$\theta_i\colon T(\Gamma_i)\to\Z^\times$ be a signature, for~$i=1,2$.
If every finitely generated subgroup~$H_i$ of~$T(\Gamma_i)$ is a T-RAAG, and the restriction~$\theta\vert_{H_i}\to\Z^\times$ is a signature, then every finitely generated subgroup~$H$ of~$T(\Gamma)$ is a T-RAAG, and the signature~$\theta\colon T(\Gamma)\to\Z^\times$ induced by~$\theta_1,\theta_2$ yields a signature~$\theta\vert_H\colon H\to\Z^\times$ of~$H$.
\end{proposition}

\subsection{Semi-direct products}

\begin{proposition}
Let~$\Gamma = (V, E, D,o,t)$ be a mixed graph, and let~$\theta \colon T(\Gamma) \to \mathbb{Z}^\times$ be a signature. Suppose that for every finitely generated subgroup~$H_0$ of~$T(\Gamma)$, the group~$H_0$ is a T-RAAG associated to some graph~$\Lambda$, and that the restriction
\[
\theta\vert_{H_0} \colon H_0 = T(\Lambda) \longrightarrow \mathbb{Z}^\times
\]
is a signature. Then every finitely generated subgroup~$H$ of the T-RAAG associated to the cone~$\nabla_\theta(\Gamma)$ is a T-RAAG associated to some graph~$\tilde\Lambda$.
Moreover, if~$\tilde\theta \colon T(\nabla_\theta(\Gamma)) \to \mathbb{Z}^\times$ is the signature induced by~$\theta$, then the restriction
\[
\tilde\theta\vert_H \colon H = T(\tilde\Lambda) \longrightarrow \mathbb{Z}^\times
\]
is a signature of~$H$.
\end{proposition}

\begin{proof}
Let~$G_0 = T(\Gamma)$ and~$G = T(\nabla_\theta(\Gamma))$. Also, let~$w$ be the tip of the cone~$\nabla(\Gamma)$, and define~$Z = \langle w \rangle$. Let~$\pi \colon G \to G_0$ denote the canonical projection. This induces the following short exact sequence:
\begin{equation}
\begin{tikzcd}
1 \arrow{r} & Z \arrow{r} & G \arrow{r}{\pi} & G_0 \arrow{r} & 1
\end{tikzcd}
\end{equation}

Thus, we have~$G = Z \rtimes_\theta G_0$, and~$\pi \vert_{G_0} = \text{id}_{G_0}$. It follows that for every~$g \in G$, 
\begin{equation}\label{eq:semidirectprop pi g}
g^{-1} \cdot \pi(g) \in \Ker(\pi) = Z.
\end{equation}

Now, let~$H \leq G$ be a finitely generated subgroup, and consider the restriction~$\pi\vert_H \colon H \to G_0$. Clearly,~$\pi(H)$ is a finitely generated subgroup of~$G_0$, and by hypothesis, it is a T-RAAG~$T(\Lambda)$, associated with a mixed graph~$\Lambda = (V_{\Lambda}, E_\Lambda, D_\Lambda, o_\Lambda, t_\Lambda)$.

Let~$V_\Lambda = \{v_1, \dots, v_r\}$. Then for each~$i = 1, \dots, r$, there exists~$h_i \in H$ such that~$v_i = \pi(h_i)$, and by \Cref{eq:semidirectprop pi g}, we have 
\[
h_i^{-1}v_i = w^{-n_i} \quad \Longrightarrow \quad h_i = v_i w^{n_i}, \qquad \text{for some } n_i \in \mathbb{Z}.
\]

\textbf{Claim.} The assignment~$\sigma(v_i) = v_i w^{n_i}$ for each~$i = 1, \dots, r$ extends to a group homomorphism~$\sigma \colon \pi(H) = T(\Lambda) \to H$. This would provide a split exact sequence:
\begin{equation}
\begin{tikzcd}
1 \arrow{r} & H \cap Z \arrow{r} & H \arrow{r}{\pi} & \pi(H) \arrow{r} \arrow[bend left=33, dashed]{l}{\sigma} & 1
\end{tikzcd}
\end{equation}

\textit{Proof of Claim.} Let~$v_i, v_j$ be a pair of joined vertices of~$\Lambda$. There are two cases to consider:
\begin{itemize}
 \item[(a)] If~$[v_i, v_j] \in E_\Lambda$, then~$v_i v_j = v_j v_i$, and moreover,~$\theta(v_i) = \theta(v_j) = 1$. Since~$\theta\vert_{\pi(H)} \colon \pi(H) = T(\Lambda) \to \mathbb{Z}^\times$ is a signature by hypothesis, we have~$w^{v_i} = w^{v_j} = w$. Thus,
 \[
  \sigma(v_i) \sigma(v_j) = (v_i w^{n_i})(v_j w^{n_j}) = v_j w^{n_j} v_i w^{n_i} = \sigma(v_j) \sigma(v_i),
 \]
 showing that~$\sigma$ respects the commutation relations.
 \item[(b)] If~$[v_i, v_j\rangle \in D_\Lambda$, then~$v_i^{v_j} = v_i^{-1}$. Moreover,~$\theta(v_i) = 1$ while~$\theta(v_j) = -1$. Hence,~$w^{v_i} = w$ and~$w^{v_j} = w^{-1}$. We have
 \[
 \begin{split}
  \sigma(v_i)^{\sigma(v_j)} & = (v_j w^{n_j}) v_i w^{n_i} (v_j w^{n_j})^{-1} \\
  & = w^{-n_j} v_j v_i w^{n_i} w^{-n_j} v_j^{-1} \\
  & = w^{-n_j} v_i^{-1} v_j w^{n_i - n_j} v_j^{-1} \\
  & = w^{-n_j} v_i^{-1} w^{n_j - n_i} \\
  & = w^{-n_i} v_i^{-1} = (v_i w^{n_i})^{-1}  \\
  & = \sigma(v_i)^{-1}.
 \end{split}
\]
Thus,~$\sigma$ respects the inversion relations.
\end{itemize}
This completes the proof of the claim.

Next, we show that~$\sigma$ is a monomorphism. Indeed, the composition of homomorphisms
\[
 \xymatrix@R=1.5pt{&H\ar[rd]^-{\pi\vert_H}& \\ \pi(H)=T(\Lambda)\ar@/_/[rr]\ar[ru]^-{\sigma} & & \pi(H)=T(\Lambda)}
\]
is the identity on~$\pi(H)$, since for every~$i = 1, \dots, r$ we have
$$\pi \circ \sigma(v_i) = \pi(v_i w^{n_i}) = \pi(v_i) = v_i,$$
as~$\pi(H) \leq G_0$ and~$\pi\vert_{G_0}$ is the identity of~$G_0$.

Therefore,~$\sigma(T(\Lambda))$ is a subgroup of~$H$ isomorphic to the T-RAAG~$T(\Lambda)$. Moreover, this proves that
$\sigma(T(\Lambda)) \cap \Ker(\pi \vert_H)$ is trivial. Hence, we can write
$$H = (H \cap Z) \rtimes \sigma(T(\Lambda)),$$
since~$H \cap Z = \Ker(\pi \vert_H)$ and~$\sigma(T(\Lambda)) \simeq \pi(H)$.

Finally, pick~$n \in \mathbb{Z}$ such that~$w^n$ generates~$H \cap Z$. Observe that~$\tilde{\theta} \vert_{\sigma(T(\Lambda))} \colon \sigma(T(\Lambda)) \to \mathbb{Z}^\times$ is a signature, as
\[
 \tilde{\theta}(\sigma(v_i)) = \tilde{\theta}(v_i w^{-n_i}) = \theta(v_i) \cdot 1 \quad \text{for every } i = 1, \dots, r,
\]
since~$\tilde{\theta}(w) = 1$ by definition, and~$\theta \vert_{\pi(H)}$ is a signature. Moreover,
\[
 \left(w^n\right)^{\sigma(v_i)} = \left(w^{v_i w^{-n_i}}\right)^n = \left(w^{\theta(v_i)}\right)^n = \left(w^n\right)^{\tilde{\theta}(\sigma(v_i))}.
\]
This shows that~$H$ is isomorphic to the T-RAAG associated to the cone~$\nabla_{\theta \vert_{\pi(H)}}(\Lambda)$ if~$n \neq 0$, and is isomorphic to~$T(\Lambda)$ otherwise. In both cases,
$$\tilde{\theta} \vert_H = (\mathbf{1}, \tilde{\theta} \vert_{\sigma(T(\Lambda))}) \colon (H \cap Z) \rtimes \sigma(T(\Lambda)) \longrightarrow \mathbb{Z}^\times$$
is a signature.
\end{proof}

%%%%%%%%%%%%%%%%%%%%%%%%%%%%%%%%%%%%%%%%%%%%%%%%%%%
%%%%%%%%%%%%%
%%%%%%%%%%%%%%%%%%%%%%%%%%%%%%%%%%%%%%%%%%%%%%%%%%5
%%%%%%%%%%%55
%%%%%%%%%%%%%%%%%%%%%%%%%%%%%%%%%%%%%%%%%%%%%%%%%%%%%%5

\section{Rigidity of T-RAAGs}\label{sec: Rigidity of T-RAAGs}

Recall from the Introduction that a mixed graph~$\Gamma=(V,E,D,o,t)$ is said to be rigid if the following occurs: if the associated T-RAAG~$T(\Gamma)$ is isomorphic to a T-RAAG~$T(\Gamma')$, associated to some mixed graph~$\Gamma'=(V',E',D',o',t')$, then~$\Gamma$ and~$\Gamma'$ are isomorphic as mixed graphs.

\begin{remark}\label{rem:simplicial rigidity}
By \citep{droms1987isomorphisms}, simplicial graphs are rigid, namely if~$\Gamma,\Gamma'$ are two simplicial graphs such that~$T(\Gamma)\simeq T(\Gamma')$, then~$\Gamma$ and~$\Gamma'$ are isomorphic as simplicial graphs.
\end{remark}

The goal of this section is to prove~\autoref{thm:isomorphism}.

%%%%%%%%%%%%%%%%%%%%%%%%%%%%%%%%

%%%%%%%%%%%%%%%%%%%%%%%%%%%%%%%%%%%%%%%%%%5
%%%%%%%%%%%%5
%%%%%%%%%%%%%%%%%%%%%%%%%%%%%%%%%%%%%%%%%%%%%%%

\subsection{Satellites}
We give the formal definition of satellite of a special mixed graph.

\begin{definition}\label{defin:satellite}
 Let~$\Gamma=(V,E,D,o,t)$ be a special mixed graph, an let~$w\in V$ be a sinkhole.
 A satellite of~$w$ is a vertex~$v\in V$,~$v\neq w$, satisfying te following:
 \begin{itemize}
  \item[(i)]~$\{v,w\}\notin E$;
  \item[(ii)] there exists at least a vertex~$u\in V$ such that~$\{v,u\},\{w,u\}\in E$;
  \item[(iii)] for every~$v'\in V$ such that~$\{v,u\}\in E$, then also~$\{w,u\}\in E$. \end{itemize}
\end{definition}
Observe that we do not require a satellite to be a positive vertex, nor a sinkhole.

First, we prove statement~(i) of~\autoref{thm:isomorphism}.

\begin{proposition}\label{prop:iso 1}
 Let~$\Gamma=(V,E,D,o,t)$ be a special mixed graph.
 If~$\Gamma$ has a sinkhole with a satellite, then it is not rigid.
\end{proposition}

\begin{proof}
Let~$w\in V$ be a sinkhole, and let~$v$ be a satellite of~$w$. One has two cases: either~$v$ is a positive vertex, or it is a sinkhole.
We construct another mixed graph~$\Gamma'=(V',E',D',o',t')$, not isomorphic to~$\Gamma$ as a mixed graph, in the former case; the construction for the latter case is analogue.

First, we observe that if~$u\in V$ is joined to~$v$, then~$u$ must be a positive vertex, as~$u$ is joined also to~$w$ and~$\Gamma$ is special.
Hence, the edge~$\{u,v\}$ is not directed.
Then we set~$v'=vw$, and moreover
\[
 \begin{split}
 V'&=(V\smallsetminus\{v\})\cup\{v'\}, \\
 E'&= (E\smallsetminus\{\:\{v,u\}\:\mid\:\{v,u\}\in E\:\})\cup\{\:\{v',u\}\:\mid\:\{v,u\}\in E\:\},\\
 D'&= D\cup \{\:[u,v'\rangle\:\mid\:\{v,u\}\in E\:\},
 \end{split}
\]
with~$o',t'\colon D'\to V'$ defined accordingly.
By the above observation, the mixed graph~$\Gamma'$ is special, as~$v'$ is a sinkhole, and every vertex joined to~$v'$ is positive.
Roughly speaking, we have "transformed" the positive vertex~$v$ into a sinkhole~$v'$; moreover, we underline that~$v'$ is a satellite of~$w$ in~$\Gamma'$.
Finally, the two mixed graphs~$\Gamma,\Gamma'$ are not isomorphic, as~$D'$ is strictly bigger than~$D$.

Now consider the map~$\psi\colon V\to V'$ given by~$\psi(v)= v'$,~$\psi(w)= w$, and~$\psi(u)=u$ for every~$u\in V$,~$u\neq v,w$.
We claim that it yields a homomorphism of groups~$\psi_T\colon T(\Gamma)\to T(\Gamma')$.
Indeed, for every edge~$\{u,u'\}\in E$ with~$u,u'\neq v$, clearly one has
\[
 \psi_T([u,u'])=[\psi(u),\psi(u')]\qquad\text{and}\qquad
 \psi_T([u,u'\rangle)=[\psi(u),\psi(u')\rangle,
\]
depending on whether~$\{u,u'\}$ is directed.
On the other hand, if~$\{v',u\}\in E$ for some~$u\in V'$, then~$[u,v]\in E$,~$[u,v'\rangle\in D'$, while~$[u,w\rangle$ lies in both~$D$ and~$D'$.
Then
\[
 [u,v']=[u,vw]=[u,w]\cdot[u,v]^w=[u,w]=u^{-2},
\]
and hence
\[
 \psi_T([u,v'\rangle)=\psi_T([u,w\rangle)=[\psi(u),\psi(w)\rangle=[\psi(u),\psi(v')\rangle.\qedhere
\]
\end{proof}

%%%%%%%%%%%%%%%%%%%%%%%%%%%%%%%%%%%%%%%%%%5
%%%%%%%%%%%%5
%%%%%%%%%%%%%%%%%%%%%%%%%%%%%%%%%%%%%%%%%%%%%%%

%%%%%%%%%%%%%%%%%%%%%%%%%%%%%%%%%%%%%%%5

\subsection{Sinkholes joined to everything}

Next, we prove statement~(ii) of~\autoref{thm:isomorphism}.
First, we need the following.

\begin{lemma}\label{lemma:special TRAAG}
  Let~$\Gamma=(V,E,D,o,t)$ be a special mixed graph.
  %, with associated T-RAAG~$T(\Gamma)$.
  If~$\Gamma'=(V',E',D',o',t')$ is a mixed graph and~$\phi\colon T(\Gamma')\to T(\Gamma)$ an isomorphism, then also~$\Gamma'$ is a special mixed graph.
\end{lemma}

\begin{proof}
Let~$\Delta=(V_\Delta,E_\Delta,D_\Delta,o_\Delta,t_\Delta)$ be a complete special mixed graph with~$|V|=|V_\Delta|$, and~$\psi_\Delta\colon T(\Gamma)\to T(\Delta)$ an epimorphism (cf.~\autoref{proposition:special}).
Then clearly~$|V'|=|V_\Delta|$, and also~$$\phi^{-1}\circ\psi_\Delta\colon T(\Gamma')\to T(\Delta)$$ is an epimorphism, so~$\Gamma'$ is a special graph by~\autoref{proposition:special}.
\end{proof}

\begin{proposition}\label{prop:iso 2}
 Let~$\Gamma=(V,E,D,o,t)$ be a special mixed graph.
If~$w\in V$ is a sinkhole such that~$[v,w\rangle\in D$ for every~$v\in V$,~$v\not= w$, then~$\Gamma$ is rigid.
\end{proposition}

\begin{proof}
 Let~$\Gamma_0$ be the induced subgraph of~$\Gamma$ whose vertices are all the vertices of~$\Gamma$ but~$w$.
Since~$\Gamma$ is special, every vertex of~$\Gamma$ different to~$w$ -- and thus, every vertex of~$\Gamma_0$ -- must be a positive vertex.
Hence,~$\Gamma_0$ has no directed edges, i.e., it is a simplicial graph, and the associated T-RAAG~$T(\Gamma_0)$ is, in fact, a RAAG.
Altogether, if~$\theta\colon T(\Gamma)\to\Z^\times$ is the signature of~$T(\Gamma)$ (observe that there is a unique possible signature as~$\Gamma$ is connected, see~\autoref{rem:signature graph}), then~$\theta(w)=-1$, the restriction~$\theta\vert_{T(\Gamma_0)}$ is constantly equal to 1, and
\begin{equation}\label{eq:T sinkcone}
T(\Gamma)=T(\Gamma_0)\rtimes_\theta \langle\:w\:\rangle.
\end{equation}

 Suppose that~$\Gamma'=(V',E',D',o',t')$ is a mixed graph whose associated T-RAAG~$T(\Gamma')$ is isomorphic to~$T(\Gamma)$.
 By~\autoref{lemma:special TRAAG}, also~$\Gamma'$ is special.
Put~$n=|V|$.
By \eqref{eq:T sinkcone},
 \[
  T(\Gamma')^{ab}\simeq T(\Gamma)^{ab}\simeq (\Z/2\Z)^{n-1}\times \Z.
 \]
Hence,~$\Gamma'$ has at least a sinkhole, and exactly~$n-1$ vertices of~$\Gamma'$ are origin of directed edges.
We claim that~$\Gamma'$ has precisely one sinkhole -- say,~$w'$ --, which is joined to every other vertex of~$\Gamma'$.
Indeed, if~$w_1',w_2'\in V'$ are sinkholes, then at least one of them must be the origin of a directed edge, contradicting the fact that~$\Gamma'$ is special.
Hence, all directed edges of~$\Gamma'$ have~$w'$ as terminal vertex, and since all vertices of~$\Gamma'$ but~$w'$ are origin of a directed edge,~$w'$ is joined to all other vertices.
Therefore, the induced subgraph~$\Gamma'_0$ of~$\Gamma'$ whose vertices are the~$n-1$ vertices origin of directed edges, is a simplicial graph, and
\[
 T(\Gamma')=T(\Gamma'_0)\rtimes_{\theta'}\langle\:w'\:\rangle,
\]
where~$\theta'\colon T(\Gamma')\to\Z^\times$ is the (only) signature of~$T(\Gamma')$.

Let~$\varphi\colon T(\Gamma)\to T(\Gamma')$ be an isomorphism, and~$\pi'\colon T(\Gamma')\to\langle w'\rangle$ the canonical projection.
Also, let
$$\bar\pi\colon T(\Gamma)\longrightarrow(\Z/2\Z)^{n-1}\times \Z\qquad\text{and}\qquad
\psi\colon(\Z/2\Z)^{n-1}\times \Z\longrightarrow \langle\:w'\:\rangle$$
be the epimorphisms induced respectively by~$T(\Gamma)^{ab}\simeq(\Z/2\Z)^{n-1}\times \Z$, and by~$\Z\simeq \langle\:w'\:\rangle$.
Then the diagram
\[
 \xymatrix{ T(\Gamma_0)\rtimes_{\theta'}\langle\:w\:\rangle \ar[rr]^-{\varphi}\ar[d]^-{\bar\pi}
&&  T(\Gamma'_0)\rtimes_{\theta'}\langle\:w'\:\rangle\ar[d]^-{\pi'}\\
 (\Z/2\Z)^{n-1}\times \Z \ar[rr]^-{\psi}&&
\langle\:w'\:\rangle}
\]
commutes.
Therefore,~$T(\Gamma_0)=\ker(\pi'\circ\phi)$, and hence~$\varphi(T(\Gamma_0))\subseteq T(\Gamma_0')$.
Since~$\varphi$ is an isomorphism, and since~$\Gamma_0$ and~$\Gamma_0'$ have the same number of vertices, one has that the restriction~$\varphi\vert_{T(\Gamma_0)}\colon T(\Gamma_0)\to T(\Gamma_0')$ is an isomorphism.
By~\autoref{rem:simplicial rigidity},~$\Gamma_0$ and~$\Gamma_0'$ are isomorphic as simplicial graphs.
Therefore,~$\Gamma$ and~$\Gamma'$ are isomorphic as mixed graphs.
\end{proof}

%%%%%%%%%%%%%%%%%%%%%%%%%%%%%%%%%%%%%%%%%%%%%%%%
%%%%%%%%%%%
%%%%%%%%%%%%%%%%%%%%%%%%%%%%%%%%%%%%%%%%%%%%%%%%%%

\subsection{Droms mixed graphs}

\begin{remark}\label{rem:droms iso subgroups}
Let~$\Gamma$ be a Droms mixed graph and let~$T(\Gamma)$ be the associated T-RAAG.
By~\autoref{thm:subgroups intro}, every finitely generated subgroup of~$T(\Gamma)$ is again a T-RAAG.
If~$T(\Gamma)\cong T(\Gamma')$ for some mixed graph~$\Gamma'$, then,
again by~\autoref{thm:subgroups intro},~$\Gamma'$ is also a Droms mixed graph, as every finitely generated subgroup of~$T(\Gamma')$ is a T-RAAG.
\end{remark}

Let~$\Gamma=(V,E,D,o,t)$ be a mixed graph, and let~$\theta\colon V\to\Z^\times$ be a signature.
For~$r\geq1$, we set the mixed graph~$\nabla^r_{\theta}(\Gamma)$ to be the result of~$r$ iterated cone obtained starting from~$\Gamma$.
For example, a special

In order to prove~\autoref{thm:isomorphism}--(iii) we need the following.

\begin{lemma}\label{lemma:normal abelian subgroup}
Let~$\Gamma=(V,E,D,o,t)$ be a Droms mixed graph with at least two vertices, and suppose that~$N \cong \mathbb{Z}^r$,~$r \geq 1$, is a maximal abelian normal subgroup of the associated T-RAAG~$T(\Gamma)$. Then, there exist vertices~$w_1, \ldots, w_r \in V$ such that~$\{w_i, v\} \in E$ for all~$v \in V \smallsetminus \{w_i\}$. Moreover, either:
\begin{enumerate}
    \item[(1)]~$\theta(w_i) = 1$ for all~$1 \leq i \leq r$,~$N = \langle w_1, \ldots, w_r \rangle$, and~$\Gamma=\nabla_{\theta_0}^r(\Gamma_0)$ for some induced subgraph~$\Gamma_0$ with signature~$\theta_0$, with~$w_1,\ldots,w_r$ the positive tips of the~$r$ iterated cones; or
    \item[(2)]~$r \geq 2$,~$\theta(w_i) = -\theta(u) = 1$ for~$1 \leq i \leq r-1$, where~$u = w_r$,~$N = \langle u^2, w_1, \ldots, w_{r-1} \rangle$, and~$\Gamma=\nabla_{\theta_0}^r(\Gamma_0)$ for some induced subgraph~$\Gamma_0$ with signature~$\theta_0$, with~$w_1,\ldots,w_r$ the positive tips of the~$r$ iterated cones
\end{enumerate}
(here~$\theta\colon T(\Gamma)\to\Z^\times$ denotes the unique signature of~$T(\Gamma)$).
In particular,~$N$ is unique.
\end{lemma}

\begin{proof}
If~$\Gamma$ is disconnected, then~$T(\Gamma)$ is a non-trivial free product, which has no non-zero abelian normal subgroup by the Kurosh theorem. Hence,~$\Gamma$ is connected. In particular,~$T(\Gamma)$ has a unique signature (cf.~\autoref{rem:signature graph}).

If~$\Gamma$ has only two vertices, the result clearly holds. Suppose now that~$\Gamma$ has at least three vertices. Since~$\Gamma$ is connected, it must be a cone, say~$\Gamma = \nabla_{\bar\theta}(\Gamma_1)$ with positive tip~$w_1$, for some signature~$\bar\theta\colon T(\Gamma_1)\to\Z^\times$ (cf.~\autoref{prop: droms elementary type}).

The intersection~$N_1 = N \cap T(\Gamma_1)$ is a maximal abelian normal subgroup of~$T(\Gamma_1)$, and hence, by induction, there are pairwise distinct central vertices~$w_2, \ldots, w_{r_1}$ in~$\Gamma_1$. Again, two cases may occur:
\begin{enumerate}
    \item[(1)] The vertices~$w_2, \ldots, w_{r_1}$ are positive and~$N_1 = \langle w_2, \ldots, w_{r_1} \rangle$. Since
 ~$$[w_1, w_i] = 1,\qquad\text{for } 2 \leq i \leq r_1,$$
    and~$\langle w_1 \rangle$ is a normal subgroup with quotient~$T(\Gamma)/\langle w_1 \rangle = T(\Gamma_1)$, we deduce~$N = \langle w_1, \ldots, w_{r_1} \rangle$ and hence~$r_1 = r$.
    \item[(2)] One of the vertices, say~$u = w_{r_1}$, is negative, and~$N = \langle u^2, w_1, \ldots, w_{r_1 - 1} \rangle$, implying~$r_1 = r - 1$.
\end{enumerate}

Case (2) only occurs when~$\Gamma$ is complete and has a single negative vertex.
\end{proof}

\begin{proposition}
Let~$\Gamma = (V, E, D, o, t)$ be a Droms mixed graph. Then,~$\Gamma$ is not rigid if, and only if,~$\Gamma$ has a satellite.
\end{proposition}

\begin{proof}
If~$\Gamma$ has a satellite, then it is not rigid by ~\autoref{thm:isomorphism}--(i) (cf.~\autoref{prop:iso 1}).

For the converse implication we argue by induction on the number of vertices~$|V|$ of~$\Gamma$.
Also, by Droms' result we may assume that~$\Gamma$ has at least a sinkhole (cf.~\autoref{rem:droms iso subgroups}), and without loss of generality we may assume that~$T(\Gamma)$ does not decompose as non-trivial free product -- i.e.,~$\Gamma$ is connected.
So, suppose~$\Gamma$ is not rigid, and~$|V|\geq3$, as for~$|V|<3$ there are no satellites, and the implication is trivial. Then~$\Gamma$ is not complete  by~\autoref{thm:isomorphism}--(ii) (cf.~\autoref{prop:iso 2}).

If~$|V| = 3$ and~$\Gamma$ is not complete, then~$\Gamma$ can only be either
\[
\bullet \leftarrow \bullet \rightarrow \bullet, \quad \text{or} \quad \bullet - \bullet \rightarrow \bullet
\]
and both mixed graphs have satellites.

Now assume~$|V| \geq 4$. Since~$\Gamma$ is connected,~$T(\Gamma)$ admits a unique signature~$\theta : T(\Gamma) \to \mathbb{Z}^\times$ (cf.~\autoref{rem:signature graph}).
By~\autoref{lemma:normal abelian subgroup}, the maximal abelian normal subgroup~$N$ is generated by a set~$\{w_1,\ldots,w_r\}$ of central vertices, all positive. Moreover,~$\Gamma = \nabla^r_{\theta_0}(\Gamma_0)$ for some induced subgraph~$\Gamma_0$ -- which is again a Droms mixed graph --, and~$T(\Gamma)/N\simeq T(\Gamma_0)$.

If~$\Gamma_0$ is not rigid, then by induction~$\Gamma_0$ has a satellite, and we are done.

If~$\Gamma_0$ is rigid, then let~$\Gamma'=(V',E',D',o',t')$ be a mixed graph, not isomorphic to~$\Gamma$, such that~$T(\Gamma)\simeq T(\Gamma')$.
By~\autoref{rem:droms iso subgroups}, also~$\Gamma'$ is a Droms mixed graph.
Thus, we may apply~\autoref{lemma:normal abelian subgroup} also to~$T(\Gamma')$: namely,~$\Gamma'=\nabla_{\theta'_0}^r(\Gamma_0')$ for some induced subgraph~$\Gamma_0'$ of~$\Gamma'$, and the positive tips~$w_1',\ldots,w_r'\in V'$ of the~$r$ recursive cones generate the abelian normal subgroup~$N$.
In particular,
$$T(\Gamma_0')\simeq T(\Gamma')/N\simeq T(\Gamma)/N\simeq T(\Gamma_0).$$
Since~$\Gamma_0$ is rigid, the two mixed graphs~$\Gamma_0$ and~$\Gamma_0'$ are isomorphic -- so, we may identify them, yet the two signatures~$\theta_0,\theta_0'\colon V(\Gamma_0)\to\Z^\times$ may not coincide.
Moreover, every vertex of~$\Gamma_0$ is joined to every positive vertex~$w_1,\ldots,w_r$, in~$\Gamma$, and analogously to every positive vertex~$w'_1,\ldots,w'_r$, in~$\Gamma'$.
Altogether, one has
\[
 V=V(\Gamma_0)\cup\{w_1,\ldots,w_r\}\simeq V(\Gamma_0)\cup\{w'_1,\ldots,w'_r\}=V'
\]
and
\[
 \begin{split}
  E&=E(\Gamma_0)\cup\{\:\{w_1,u\},\ldots,\{w_r,u\}\:\mid\:u\in V(\Gamma_0)\:\} \\ &\simeq E(\Gamma_0)\cup\{\:\{w'_1,u\},\ldots,\{w'_r,u\}\:\mid\:u\in V(\Gamma_0)\:\}=E',
 \end{split}
\] hence one must have~$D\not\simeq D'$.
Since
\[
\begin{split}
D &=D(\Gamma_0)\cup\{\:[w_1,u\rangle,\ldots,[w_r,u\rangle\:\mid\:u\in V(\Gamma_0)\text{ and }\theta_0(u)=-1\:\},\\
D' &=D(\Gamma_0)\cup\{\:[w'_1,u\rangle,\ldots,[w'_r,u\rangle\:\mid\:u\in V(\Gamma_0)\text{ and }\theta'_0(u)=-1\:\},
\end{split}\]
one has~$D\not\simeq D'$ if, and only if,~$\theta_0(u)\neq\theta_0'(u)$ for some~$u\in V(\Gamma_0)$.
This is possible only if~$u$ is an isolated vertex of~$\Gamma_0$ (cf.~\autoref{rem:signature graph}).
So, pick such a vertex~$u$. One has three cases.
\begin{itemize}
 \item[(1)] If~$\theta_0(u)=1$, then~$u$ is a positive vertex of~$\Gamma$.
Since~$w_1,\ldots,w_r$ are positive vertices as well, and since we are assuming that~$\Gamma$ has at least a sinkhole, such a sinkhole -- say,~$v$ -- must be a vertex of~$\Gamma_0$. Since the only vertices of~$\Gamma$ which are joined to~$u$ are~$w_1,\ldots,w_r$, which are joined to the sinkhole~$v$ too,~$u$ is a satellite of~$v$ in~$\Gamma$.
\item[(2)] If~$\theta_0(u)=-1$ and~$\theta_0(v)=-1$ for some other vertex~$v$ of~$\Gamma_0$, then~$u$ is a satellite of~$v$ in~$\Gamma$, as all vertices that are joined to~$u$ -- i.e.,~$w_1,\ldots,w_r$ -- are joined to~$v$ too.
\item[(3)] If~$\theta_0(u)=-1$ and~$\theta_0(v)=1$ for all~$v\in V(\Gamma_0)$,~$v\neq u$, then we claim that at least one of such positive vertices -- say,~$v_0$ -- is isolated in~$\Gamma_0$, so that~$v_0$ is a satellite of~$u$ in~$\Gamma$.
Indeed, if none of such vertices is isolated, then necessarily one has~$\theta_0'(v)=1=\theta_0'(u)$ for all~$v\in V(\Gamma_0)\smallsetminus\{u\}$ (cf.~\autoref{rem:signature graph}). Therefore,~$\Gamma'$ has no directed edges -- namely, it is a simplicial graph --, but then~$\Gamma$ is rigid by Droms (cf.~\autoref{rem:simplicial rigidity}), a contradiction.
\end{itemize}
In every case, the Droms mixed graph~$\Gamma$ has a satellite, and this completes the proof.
\end{proof}

\begin{example}
 The Droms mixed graph~$\Gamma_3$ displayed in \autoref{ex:gamma3 droms} is not rigid, as the vertex~$a_2$ is a satellite of the vertex~$a_1$, and conversely; the T-RAAG~$T(\Gamma_3)$ is isomorphic to the T-RAAG~$T(\Gamma_3')$ based on the Droms mixed graph
 
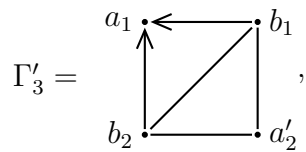
\begin{figure}[H]
\centering
\begin{tikzpicture}[>={Straight Barb[length=7pt,width=6pt]},thick, scale =0.75]

\draw[] (-5.8, 0) node[left] {$\Gamma'_3 =~$};
\draw[fill=black] (-5,1) circle (1pt) node[left] {$a_1$};
\draw[fill=black] (-3,1) circle (1pt) node[right] {$b_1$};
\draw[fill=black] (-3,-1) circle (1pt) node[right] {$a'_2$};
\draw[fill=black] (-5,-1) circle (1pt) node[left] {$b_2$};
\draw[] (-2.5, 0) node[right] {$,$};

\draw[thick, ->] (-3.1, 1) -- (-4.9, 1);
\draw[thick] (-3,0.9) -- (-3,-0.9);
\draw[thick] (-4.9,-0.9) -- (-3.1,0.9);
\draw[thick, ->] (-5,-0.9) -- (-5,0.9);
\draw[thick] (-4.9,-1) -- (-3.1,-1);
\end{tikzpicture}
\end{figure}
which is not isomorphic to~$\Gamma_3$.

 The Droms graph~$\Upsilon$ displayed in \autoref{ex:upsilon droms} is rigid, as it has no satellites: indeed, the only sinkhole is~$a_1$, and the vertices with distance 2 to~$a_1$ are the positive vertices~$a_2,b_2$, which are joined together, so that none of them is a satellite of~$a_1$.
\end{example}

\noindent\textit{\\ Islam Foniqi,\\
The University of East Anglia\\ 
Norwich (United Kingdom)\\}
{email: i.foniqi@uea.ac.uk}

\noindent\textit{\\ Claudio Quadrelli,\\
University of Insubria\\
Como (Italy EU)\\}
{email: claudio.quadrelli@uninsubria.it}

\setlength{\bibsep}{7pt}
\bibliography{main}

\begin{thebibliography}{14}
\providecommand{\natexlab}[1]{#1}
\providecommand{\url}[1]{\texttt{#1}}
\expandafter\ifx\csname urlstyle\endcsname\relax
  \providecommand{\doi}[1]{doi: #1}\else
  \providecommand{\doi}{doi: \begingroup \urlstyle{rm}\Url}\fi

\bibitem[Antol{\'\i}n et~al.(2025)Antol{\'\i}n, Blufstein, and
  Paris]{antolin2025traag}
Y.~Antol{\'\i}n, M.~Blufstein, and L.~Paris.
\newblock When is a traag orderable?
\newblock \emph{Communications in Algebra}, pages 1--5, 2025.

\bibitem[Blumer et~al.(2024)Blumer, Quadrelli, and Weigel]{blumer2024oriented}
S.~Blumer, C.~Quadrelli, and T.~S. Weigel.
\newblock Oriented right-angled artin pro-$l$ groups and maximal pro-$l$ galois
  groups.
\newblock \emph{International Mathematics Research Notices}, 2024\penalty0
  (8):\penalty0 6790--6819, 2024.

\bibitem[Charney(2007)]{charney2007introduction}
R.~Charney.
\newblock An introduction to right-angled {Artin} groups.
\newblock \emph{Geometriae Dedicata}, 125\penalty0 (1):\penalty0 141–158,
  2007.

\bibitem[Clancy and Ellis(2010)]{clancy2010homology}
M.~Clancy and G.~Ellis.
\newblock Homology of some {Artin} and twisted {Artin} groups.
\newblock \emph{Journal of K-Theory}, 6\penalty0 (1):\penalty0 171–196, 2010.

\bibitem[Droms(1987{\natexlab{a}})]{droms1987graph}
C.~Droms.
\newblock Graph groups, coherence, and three-manifolds.
\newblock \emph{J. ALGEBRA.}, 106\penalty0 (2):\penalty0 484--489,
  1987{\natexlab{a}}.

\bibitem[Droms(1987{\natexlab{b}})]{droms1987isomorphisms}
C.~Droms.
\newblock Isomorphisms of graph groups.
\newblock \emph{Proceedings of the American Mathematical Society}, 100\penalty0
  (3):\penalty0 407–408, 1987{\natexlab{b}}.

\bibitem[Droms(1987{\natexlab{c}})]{droms1987subgroups}
C.~Droms.
\newblock Subgroups of graph groups.
\newblock \emph{J. ALGEBRA.}, 110\penalty0 (2):\penalty0 519--522,
  1987{\natexlab{c}}.

\bibitem[Duchamp and Krob(1992)]{duchamp1992lower}
G.~Duchamp and D.~Krob.
\newblock The lower central series of the free partially commutative group.
\newblock In \emph{Semigroup Forum}, volume~45, pages 385--394. Springer, 1992.

\bibitem[Foniqi(2022)]{foniqi2022}
I.~Foniqi.
\newblock \emph{Results on {Artin} and twisted {Artin} groups}.
\newblock PhD thesis, University of Milano - Bicocca, 2022.
\newblock URL \url{https://boa.unimib.it/handle/10281/374264}.

\bibitem[Foniqi(2024{\natexlab{a}})]{foniqi2024subgroup}
I.~Foniqi.
\newblock Subgroup separability of twisted right-angled artin groups.
\newblock \emph{arXiv preprint arXiv:2410.06914}, 2024{\natexlab{a}}.

\bibitem[Foniqi(2024{\natexlab{b}})]{foniqi2024twisted}
I.~Foniqi.
\newblock Twisted right-angled artin groups.
\newblock \emph{arXiv preprint arXiv:2407.06933}, 2024{\natexlab{b}}.

\bibitem[Himeno and Teragaito(2024)]{himeno2024twisted}
K.~Himeno and M.~Teragaito.
\newblock Twisted right-angled artin groups embedded in knot groups.
\newblock \emph{arXiv preprint arXiv:2412.03849}, 2024.

\bibitem[Karrass and Solitar(1970)]{karrass1970subgroups}
A.~Karrass and D.~Solitar.
\newblock The subgroups of a free product of two groups with an amalgamated
  subgroup.
\newblock \emph{Trans. Amer. Math. Soc.}, 150:\penalty0 227--255, 1970.
\newblock ISSN 0002-9947.
\newblock \doi{10.2307/1995492}.
\newblock URL \url{https://doi.org/10.2307/1995492}.

\bibitem[Pride(1986)]{pride1986tits}
S.~J. Pride.
\newblock On tits' conjecture and other questions concerning artin and
  generalized artin groups.
\newblock \emph{Inventiones mathematicae}, 86\penalty0 (2):\penalty0 347--356,
  1986.

\end{thebibliography}

%\bibliographystyle{plain}
%\end{thebibliography}

\end{document}